\documentclass[11pt]{amsart}
\usepackage{amsopn}
\usepackage{amssymb, amscd, amsmath}
\usepackage{pstricks-add}
\usepackage[labelsep=period,indention=1.8cm,labelfont=bf,
textfont=it,justification=centering]{caption}%font=small
\usepackage{enumerate}

\newpsobject{grilla}{psgrid}{subgriddiv=10,griddots=10,gridlabels=6pt}
\newcommand{\nc}{\newcommand}

\nc{\vg}{\mathfrak{v} } \nc{\wg}{\mathfrak{w} } \nc{\zg}{\mathfrak{z} }
\nc{\ngo}{\mathfrak{n} } \nc{\kg}{\mathfrak{k} } \nc{\mg}{\mathfrak{m} }
\nc{\bg}{\mathfrak{b} } \nc{\ggo}{\mathfrak{g} } \nc{\ggob}{\overline{\mathfrak{g}}
} \nc{\sog}{\mathfrak{so} } \nc{\sug}{\mathfrak{su} } \nc{\spg}{\mathfrak{sp} }
\nc{\slg}{\mathfrak{sl} } \nc{\glg}{\mathfrak{gl} } \nc{\cg}{\mathfrak{c} }
\nc{\rg}{\mathfrak{r} } \nc{\hg}{\mathfrak{h} } \nc{\tg}{\mathfrak{t} }
\nc{\ug}{\mathfrak{u} } \nc{\dg}{\mathfrak{d} } \nc{\ag}{\mathfrak{a} }
\nc{\pg}{\mathfrak{p} } \nc{\sg}{\mathfrak{s} } \nc{\pca}{\mathcal{P}}
\nc{\nca}{\mathcal{N}} \nc{\lca}{\mathcal{L}} \nc{\oca}{\mathcal{O}}
\nc{\mca}{\mathcal{M}} \nc{\tca}{\mathcal{T}} \nc{\aca}{\mathcal{A}}
\nc{\cca}{\mathcal{C}} \nc{\sca}{\mathcal{S}} \nc{\bca}{\mathcal{B}}

\nc{\vp}{\varphi} \nc{\ddt}{{\small \frac{{\rm d}}{{\rm d}t}}} \nc{\im}{\mathtt{i}}

\nc{\SO}{{\mathrm SO}} \nc{\Spe}{{\mathrm Sp}} \nc{\Sl}{{\mathrm SL}}
\nc{\SU}{{\mathrm SU}} \nc{\Or}{{\mathrm O}} \nc{\U}{{\mathrm U}} \nc{\Gl}{{\mathrm
GL}} \nc{\Se}{{\mathrm S}} \nc{\Cl}{{\mathrm Cl}} \nc{\Spein}{{\mathrm Spin}}
\nc{\Pin}{{\mathrm Pin}}

\nc{\RR}{{\Bbb R}} \nc{\HH}{{\Bbb H}} \nc{\CC}{{\Bbb C}} \nc{\ZZ}{{\Bbb Z}}
\nc{\FF}{{\Bbb F}} \nc{\NN}{{\Bbb N}} \nc{\QQ}{{\Bbb Q}} \nc{\PP}{{\Bbb P}}
\nc{\G}{\mathrm{GL}_n(\RR)}
\nc{\preq}{\simeq_\RR}
\nc{\prek}{\simeq_K}

\nc{\vs}{\vspace{.2cm}} \nc{\vsp}{\vspace{1cm}} \nc{\ip}{\langle\cdot,\cdot\rangle}
\nc{\la}{\langle} \nc{\ra}{\rangle} \nc{\unm}{\frac{1}{2}} \nc{\unc}{\frac{1}{4}}
\nc{\und}{\frac{1}{16}} \nc{\no}{\vs\noindent} \nc{\lam}{\Lambda^2\ngo^*\otimes\ngo}
\nc{\tangz}{{\rm T}^{\rm Zar}} \nc{\nor}{{\sf n}}
\nc{\eigen}{(k_1<...<k_r;d_1,...,d_r)} \nc{\eigencero}{(0<k_2<...<k_r;d_1,...,d_r)}
\nc{\mum}{/\!\!/} \nc{\kir}{/\!\!/\!\!/}

\nc{\He}{\operatorname{Hess}} \nc{\ad}{\operatorname{ad}}
\nc{\Ad}{\operatorname{Ad}} \nc{\rank}{\operatorname{rank}}
\nc{\Irr}{\operatorname{Irr}} \nc{\End}{\operatorname{End}}
\nc{\Aut}{\operatorname{Aut}} \nc{\Inn}{\operatorname{Inn}}
\nc{\Der}{\operatorname{Der}} \nc{\Ker}{\operatorname{Ker}}
\nc{\Iso}{\operatorname{I}} \nc{\Diff}{\operatorname{Diff}}
\nc{\Lie}{\operatorname{L}} \nc{\tr}{\operatorname{tr}} \nc{\dif}{\operatorname{d}}
\nc{\sen}{\operatorname{sen}} \nc{\modu}{\operatorname{mod}}
\nc{\Ric}{\operatorname{Ric}} \nc{\Ricac}{\operatorname{Ric^{ac}}}
\nc{\Ricg}{\operatorname{Ric^{\gamma}}} \nc{\Ricc}{\operatorname{Ric^{c}}}
\nc{\Ricj}{\operatorname{Ric^{J}}}
\nc{\sym}{\operatorname{sym}} \nc{\symac}{\operatorname{sym^{ac}}}
\nc{\symc}{\operatorname{sym^{c}}} \nc{\scalar}{\operatorname{sc}}
\nc{\grad}{\operatorname{grad}} \nc{\ricci}{\operatorname{Rc}}
\nc{\ricciac}{\operatorname{ric^{ac}}} \nc{\riccic}{\operatorname{ric^{c}}}
\nc{\riccig}{\operatorname{ric^{\gamma}}} \nc{\Rin}{\operatorname{M}}
\nc{\Le}{\operatorname{L}} \nc{\tang}{\operatorname{T}}
\nc{\level}{\operatorname{level}} \nc{\rad}{\operatorname{r}}
\nc{\abel}{\operatorname{ab}}
\nc{\Pf}{\operatorname{Pf}}

\theoremstyle{plain}
\newtheorem{theorem}{Theorem}[section]
\newtheorem{proposition}[theorem]{Proposition}
\newtheorem{corollary}[theorem]{Corollary}

\theoremstyle{definition}
\newtheorem{definition}[theorem]{Definition}

\theoremstyle{remark}
\newtheorem{remark}[theorem]{Remark}

\newtheorem{example}[theorem]{Example}

\title[Invariants of complex structures]{Invariants of complex structures on nilmanifolds}

\author{Edwin Alejandro Rodr\'{\i}guez Valencia}

\address{FaMAF and CIEM, Universidad Nacional de C\'ordoba, Medina Allende s/n, 5000 C\'ordoba, Argentina}
\email{earodriguez@famaf.unc.edu.ar}

\thanks{
Fully supported by a CONICET fellowship (Argentina).}

\begin{document}
\renewcommand{\tablename}{Table}
\maketitle

\begin{abstract}
Let $(N, J)$ be a simply connected $2n$-dimensional nilpotent Lie group \linebreak endowed
with an invariant complex structure. We define a left invariant Riemannian metric
on $N$ compatible with $J$ to be {\it minimal}, if it minimizes the norm of the
invariant part of the Ricci tensor among all compatible metrics with the same
scalar curvature.  In \cite{canonical}, J. Lauret proved that minimal
metrics (if any) are unique up to isometry and sca\-ling.  This uniqueness allows
us to distinguish two complex structures with Riemannian data, giving rise to a
great deal of invariants.

We show how to use a Riemannian invariant: the eigenvalues of the
Ricci operator, polynomial invariants and discrete invariants to give an alternative
proof of the pairwise non-isomorphism between the structures which have appeared in the
classification of abelian complex structures on $6$-dimensional nilpotent Lie algebras
given in \cite{AndBrbDtt}.  We also present some continuous families in dimension $8$.
\end{abstract}

%=================================================================================

\section{Introduction}\label{intro}

Let $N$ be a real $2n$-dimensional nilpotent Lie group with Lie algebra $\ngo$,
whose Lie bracket will be denoted by $\mu :\ngo\times\ngo\longrightarrow\ngo$.  An {\it invariant complex structure} on $N$ is defined by a map $J:\ngo\longrightarrow\ngo$ satisfying $J^2=-I$ and the integrability condition
\begin{equation}\label{integral}
\mu(JX,JY)=\mu(X,Y)+J\mu(JX,Y)+J\mu(X,JY), \qquad \forall X,Y\in\ngo.
\end{equation}
By left translating $J$, one obtains a complex manifold $(N,J)$, as well as compact complex manifolds $(N/\Gamma,J)$ if $N$ admits cocompact discrete subgroups $\Gamma$, which are usually called \textit{nilmanifolds} and play an important role in complex geometry.

The automorphism group $\Aut(\ngo)$ acts by conjugation on the set of all invariant complex structures on $\ngo$, and hence two such structures are considered to be {\it equivalent} if they belong to the same conjugation class.  The lack of invariants makes the classification of invariant complex structures a difficult task.  This has only been achieved in dimension $\leq 6$ in the nilpotent case in \cite{nilstruc}, and for any $6$-dimensional Lie algebra in the abelian case in \cite{AndBrbDtt}.

Our aim in this paper is to use two different invariants (namely, minimal metrics and Pfaffian forms, see below), to give an alternative proof of the non-equivalence between any two abelian complex structures on nilpotent Lie algebras of dimension $6$ obtained in the classification list given in \cite[Theorem 3.5.]{AndBrbDtt}.  Along the way, we prove that any such structure, excepting only one, does admit a minimal metric.  As another application of the invariants, we give in Section \ref{result8} many families depending on one, two and three parameters of abelian complex structures on $8$-dimensional $2$-step nilpotent Lie algebras, showing that a full classification could be really difficult in dimension $8$.

%---------------------------------------------------------------------------------

\subsection{Minimal metrics} A left invariant metric which is
{\it compatible} with $(N,J)$, also called a {\it hermitian metric}, is determined by an inner product $\ip$ on $\ngo$ such that
$$
\la JX,JY\ra=\la X,Y\ra, \quad \forall X,Y\in\ngo.
$$
We consider
$$
\Ricc_{\ip}:=\unm\left(\Ric_{\ip}-J\Ric_{\ip}J\right),
$$
the complexified part of the Ricci operator $\Ric_{\ip}$ of the Hermitian manifold $(N,J,\ip)$, and the corresponding
$(1,1)$-component of the Ricci tensor $\riccic_{\ip}:=\la\Ricc_{\ip}\cdot,\cdot\ra$.

A compatible metric $\ip$ on $(N,J)$ is called {\it minimal} if
$$
\tr{(\Ricc_{\ip})^2}=\min \left\{ \tr{(\Ricc_{\ip'})^2} :
\scalar(\ip')=\scalar(\ip)\right\},
$$
where $\ip'$ runs over all compatible metrics on $(N,J)$ and $\scalar(\ip)=\tr\Ric_{\ip}=\tr\Ricc_{\ip}$ is the scalar curvature.  In \cite{canonical}, the following conditions on $\ip$ are proved to be equivalent to minimality, showing that such
metrics are special from many other points of view:
\begin{itemize}
\item[(i)] The solution $\ip_t$ with initial value $\ip_0=\ip$ to the evolution equation
$$
\ddt \ip_t=-2\riccic_{\ip_t},
$$
is self-similar, in the sense that $\ip_t=c_t\vp_t^*\ip$ for some $c_t>0$ and one-parameter group
of automorphisms $\vp_t$ of $N$.

\item[(ii)] There exist a vector field $X$ on $N$ and $c\in\RR$ such that
$$
\riccic_{\ip}=c\ip+L_X\ip,
$$
where $L_X\ip$ denotes the usual Lie derivative.  In analogy with the well-known concept in Ricci flow theory,
one may call $\ip$ a $(1,1)$-Ricci soliton.

\item[(iii)] $\Ricc_{\ip}=cI+D$ for some $c\in\RR$ and
$D\in\Der(\ngo)$.
\end{itemize}

The uniqueness up to isometric isomorphism and scaling of a minimal metric on a
given $(N,J)$ was also proved in \cite{canonical}, and can be used to obtain invariants in the following way.  If $(N,J_1,\ip_1)$ and $(N,J_2,\ip_2)$ are minimal and $J_1$ is equivalent to $J_2$, then they must be conjugate via an automorphism which is an isometry between $\ip_1$ and $\ip_2$.  This provides us with a lot of invariants, namely the Riemannian geometry invariants including all different kind of curvatures.

%---------------------------------------------------------------------------------

\subsection{Pfaffian forms}
Consider a real vector space $\ngo$ and fix a direct sum decomposition
$$
\ngo=\ngo_1\oplus\ngo_2, \qquad \dim{\ngo_1}=m, \qquad\dim{\ngo_2}=n.
$$
Every
$2$-step nilpotent Lie algebra of dimension $m+n$ with derived algebra of dimension $\leq n$ can be represented by a bilinear skew-symmetric map
$$
\mu:\ngo_1\times\ngo_1\longrightarrow\ngo_2.
$$

For a given inner product $\ip$ on $\ngo=\ngo_1\oplus\ngo_2$ (with $\ngo_1\perp\ngo_2$), one can encode the structural constants of $\mu$ in a map $J_\mu:\ngo_2\longrightarrow\sog(\ngo_1)$ defined by
$$
\la J_\mu(Z)X,Y\ra = \la\mu(X,Y),Z\ra, \qquad \forall X,Y\in\ngo_1, \; Z\in\ngo_2.
$$
There is a nice and useful isomorphism invariant for $2$-step algebras (with $m$ even) called the {\it Pfaffian form}, which is the projective equivalence class of the homogeneous polynomial $f_\mu$ of degree $m/2$ in $n$ variables defined by
$$
f_\mu(Z)^2=\det{J_\mu(Z)}, \qquad \forall Z\in\ngo_2,
$$
for each $\mu$ of type $(n,m)$ (see Section \ref{Pfaffian}).

For each $\mu\in V_{n,m}:=\Lambda^2\ngo_{1}^*\otimes\ngo_{2}$, let $N_\mu$ denote the simply connected nilpotent Lie group with Lie algebra $(\ngo, \mu)$. We prove that if two complex nilmanifolds $(N_{\mu}, J)$ and $(N_{\lambda}, J)$ are \textit{holomorphically isomorphic}, then
$f_\lambda\in \RR_{>0}\Gl_{q}(\CC)\cdot f_\mu$, with $n=2q$ (see Proposition \ref{PfafCC}).  This will allow us to use the existence of minimal metrics to distinguish complex nilmanifolds by means of invariants of forms.

\vs \noindent {\it Acknowledgements.}  This research is part of the Ph.D. thesis (Universidad Nacional de
C\'ordoba) by the author. I am grateful to my advisor Jorge Lauret for his invaluable help
during the preparation of the paper.

%=================================================================================

\section{Preliminaries}\label{basic}

In this section, we recall basic notions on complex structures on nilmanifolds
and their \linebreak hermitian metrics.

Let $N$ be a simply connected $2n$-dimensional nilpotent Lie group with Lie algebra $\ngo$,
whose Lie bracket will be denoted by $\mu :\ngo\times\ngo \to \ngo$.
\ An invariant {\it complex structure} on $N$ is defined by a map
$J:\ngo\to\ngo$ satisfying $J^2=-I$ and such that
\begin{equation*}%\label{integral}
\mu(JX,JY)=\mu(X,Y)+J\mu(JX,Y)+J\mu(X,JY), \qquad \forall X,Y\in\ngo.
\end{equation*}
We say that $J$ is \textit{abelian} if the following condition holds:
$$\mu(JX,JY)=\mu(X,Y), \qquad \forall X,Y\in\ngo.$$

\begin{definition}\label{holomiso}
Two complex structures $J_1$ and $J_2$ on $N$ are said to be \textit{equivalent} if there
exists an automorphism $\alpha$ of $\ngo$ satisfying $J_2 = \alpha J_1 \alpha^{-1}$. \
Two pairs $(N_1, J_1)$ and $(N_2, J_2)$ are \textit{holomorphically isomorphic} if
there exists a Lie algebra isomorphism $\alpha: \ngo_1 \to \ngo_2$ such that
$J_2 = \alpha J_1 \alpha^{-1}$.
\end{definition}

We fix a $2n$-dimensional real vector space $\ngo$, and consider as a parameter space for
the set of all real nilpotent Lie algebras of a given dimension $2n$, the algebraic subset
$$
\nca:=\{\mu\in V:\mu\;\mbox{satisfies Jacobi and is nilpotent}\},
$$
where $V:=\lam$ is the vector space of all skew-symmetric bilinear maps from $\ngo\times\ngo$
to $\ngo$. Recall that any inner product $\ip$ on $\ngo$ determines an inner product
on $V$, also denoted by $\ip$, as follows: if $\{e_i\}$ is a orthonormal basis
of $\ngo$,
\begin{align}\label{prodtint}
\langle \mu, \lambda \rangle& := \sum_{i,j} \langle \mu(e_i,e_j), \lambda(e_i,e_j) \rangle\\
&=\sum_{i,j,k} \langle \mu(e_i,e_j), e_k \rangle \langle \lambda(e_i,e_j), e_k \rangle\notag.
\end{align}

For each $\mu\in\nca$, let $N_\mu$ denote the simply connected nilpotent Lie group
with Lie algebra $(\ngo, \mu)$. We now fix a map $J:\ngo\to\ngo$ such that $J^2=-I$. The
corresponding Lie group $$\Gl_n(\CC)=\{g\in\Gl_{2n}(\RR) : gJ = Jg\}$$
acts naturally on $V$ by $g\cdot\mu(\cdot,\cdot)=g\mu(g^{-1}\cdot,g^{-1}\cdot)$,
leaving $\nca$ invariant, as well as the algebraic subset $\nca_{J}\subset \nca$ given by
$$\nca_{J}:=\{\mu\in\nca: \mu\;\mbox{satisfies (\ref{integral})}\}.$$

We can identify each $\mu\in\nca_{J}$ with a \textit{complex nilmanifold} as follows:
\begin{align}\label{metrJ}
\mu\leftrightarrow (N_{\mu}, J).
\end{align}

\begin{proposition}
Two complex nilmanifolds $\mu$ and $\lambda$ are holomorphically isomorphic if and only if
$\lambda\in \Gl_{n}(\CC)\cdot \mu$.
\end{proposition}
\begin{proof}
If we suppose that $(N_{\mu}, J)$ and $(N_{\lambda}, J)$ are holomorphically isomorphic,
then there exists a Lie algebra isomorphism $g^{-1}: (\ngo, \lambda)\mapsto (\ngo, \mu)$
such that $J=g J g^{-1}$.  Hence, $\lambda = g\cdot\mu$ and $g\in \Gl_{n}(\CC)$
(taking their matrix representation).
\end{proof}

A left invariant Riemannian metric on $N$ is said to be {\it compatible} with
a complex structure $J$ on $N$ if it is defined by an inner product $\ip$ on $\ngo$ such that
$$
\la JX,JY\ra=\la X,Y\ra, \qquad \forall X,Y\in\ngo,
$$
that is, $J$ is orthogonal with respect to $\ip$. We denote by $\cca=\cca(N,J)$ the set
of all left invariant metrics on $N$ compatible with $J$.

\begin{definition}
Two triples $(N_1, J_1, \ip)$ and $(N_2, J_2, \ip^\prime)$, with $\ip\in\cca(N_1,J_1)$ and
$\ip^\prime\in\cca(N_2,J_2)$, are said to be \textit{isometric isomorphic} if there exists a Lie algebra
\linebreak isomorphism $\varphi: \ngo_1 \to \ngo_2$ such that
$J_2 = \varphi J_1 \varphi^{-1}$ and $\ip^\prime = \langle \varphi^{-1}\cdot, \varphi^{-1}\cdot \rangle$.
\end{definition}

We now identify each $\mu\in\nca_{J}$ with a \textit{Hermitian nilmanifold} in the following way:
\begin{align}\label{metrcorch}
\mu\leftrightarrow (N_{\mu}, J, \ip),
\end{align}
where $\ip$ is a fixed inner product on $\ngo$ compatible with $J$. Therefore, each
$\mu\in\nca_{J}$ can be viewed in this way as a \textit{Hermitian metric} compatible with
$(N_{\mu}, J)$, and two metrics $\mu$, $\lambda$ are compatible with the same complex structure
if and only if they live in the same $\Gl_{n}(\CC)$-orbit. Indeed, each $g\in \Gl_{n}(\CC)$
determines a Riemannian isometry preserving the complex structure
\begin{align}\label{isom}
(N_{g\cdot\mu}, J, \ip)\to (N_{\mu}, J, \la g\cdot, g\cdot\ra)
\end{align}
by exponentiating the Lie algebra isomorphism $g^{-1}: (\ngo, g\cdot\mu)\mapsto (\ngo, \mu)$.
We then have the identification $\Gl_{n}(\CC)\cdot\mu = \cca(N_{\mu}, J)$, for any $\mu\in\nca_{J}$.

%========================================================================================================

\section{Invariants}\label{Inv}

We now discuss the problem of distinguishing two complex nilmanifolds up to holomorphic
isomorphism, by considering different types of invariants.

\subsection{Minimal metrics}\label{minimal}
In \cite{canonical}, J. Lauret showed how to use the complexified part of the Ricci operator of
a nilpotent Lie group given, to determinate the existence of compatible \textit{minimal}
metrics with an invariant geometric structure on the Lie group. Furthermore, he proved
that these metrics (if any) are unique up to isometry and scaling. This property allows us
to distinguish two geometric structure with invariants coming from Riemannian geometric.
In this section, we will be apply these results to the complex case and use the identifications
(\ref{metrJ}) and (\ref{metrcorch}) to rewrite them in terms of data arising from the Lie
algebra; this will be the basis of our method: fix a complex structure and move the
bracket. This method is explained in a more detailed way in Section \ref{result6} in the
$6$-dimensional case.

The following theorem was obtained by using strong results from geometric invariant theory,
mainly related to the moment map of a real representation of a real reductive Lie group.

\begin{theorem}\cite{canonical}\label{RS1}
Let $F: \nca_{J} \to \RR$ be defined by $F(\mu):= \tr(\Ricc_{\mu})^2/\|\mu\|^4$, where
$\Ricc_{\mu}$ is the orthogonal projection of the Ricci operator $\Ric_{\mu}$ of the
Riemannian manifold $(N_\mu,\ip)$ onto the space of symmetric maps of $(\ngo,\ip)$
which commute with $J$. Then for $\mu\in\nca_{J}$, the following conditions are equivalent:
\begin{itemize}
\item[(i)] $\mu$ is a critical point of $F$.

\item[(ii)] $F|_{\Gl_{n}(\CC)\cdot\mu}$ attains its minimum value at $\mu$.

\item[(iii)] $\Ricc_{\mu}=cI+D$ for some $c\in \RR$, $D\in\Der(\ngo)$.
\end{itemize}
Moreover, all the other critical points of $F$ in the orbit $\Gl_{n}(\CC)\cdot\mu$
lie in $\RR^{*}\U(n)\cdot\mu$.
\end{theorem}

A complex nilmanifold $\mu$ is said to be \textit{minimal} if it satisfies any of the conditions
in the previous theorem.

\begin{corollary}\label{CRS1}
Two minimal complex nilmanifolds $\mu$ and $\lambda$ are
isomorphic if and only if \linebreak $\lambda\in\RR^{*}\U(n)\cdot \mu$.
\end{corollary}

Let $(N,J,\ip)$ be a Hermitian nilmanifold, i.e. $J$ is an invariant complex structure
on $N$ and $\ip\in\cca(N,J)$.
\begin{definition}
Let $\Ric_{\ip}$ be the Ricci operator of $(N, \ip)$. The \textit{Hermitian Ricci operator}
is given by
$$
\Ricc_{\ip}:=\unm\left(\Ric_{\ip}-J\Ric_{\ip}J\right).
$$
\end{definition}

A metric $\ip\in \cca$ is called {\it minimal} if it minimizes the functional $\tr(\Ricc_{\ip})^2$
on the set of all compatible metrics with the same scalar curvature. We now rewrite
Theorem \ref{RS1} in geometric terms, by using the identification (\ref{metrcorch}).

\begin{theorem}\cite{canonical}\label{RS}
For $\ip\in \cca$, the following conditions are equivalent:
\begin{itemize}
\item[(i)] $\ip$ is minimal.

\item[(ii)] $\Ricc_{\ip}=cI+D$ for some $c\in \RR$, $D\in\Der(\ngo)$.
\end{itemize}
Moreover, there is at most one compatible left invariant metric on $(N, J)$ up to isometry
(and scaling) satisfying any of the above conditions.
\end{theorem}

Let $\ip\in \cca$ be a minimal metric with $\Ricc_{\ip}=cI+D$ for some $c\in \RR$, $D\in\Der(\ngo)$. \
We say that $\mu$ is of \textit{type} $\eigen$ if $\{k_i\}\subset \ZZ_{\geq 0}$ are the eigenvalues
of $D$ with multiplicities $\{d_i\}$  respectively and $\gcd(k_1,\ldots,k_r)=1$.

\begin{corollary}\cite{canonical}\label{CRS}
Let $J_1$, $J_2$ be two complex structures on $N$, and assume that they admit minimal
compatible metrics $\ip$ and $\ip'$, respectively. Then $J_1$ is equivalent to $J_2$ if and
only if there exists $\vp\in \Aut(\ngo)$ and $c>0$ such that $J_2 = \vp J_1 \vp^{-1}$ and
$$\la \vp X,\vp Y \ra' = c \la X, Y \ra, \quad \forall X,Y \in \ngo.$$
In particular, if $J_1$ and $J_2$ are equivalent, then their respective minimal compatible
metrics are necessarily isometric up to scaling.
\end{corollary}

By (\ref{metrcorch}) and (\ref{isom}), it is easy to see that two Hermitian nilmanifolds $\mu$
and $\lambda$ are isometric (i.e. if $(N_\mu, J, \ip)$ and $(N_\lambda, J, \ip)$ are isometric
isomorphic) if and only if they live in the same $\U(n)$-orbit. Corollary \ref{CRS} and
(\ref{metrcorch}) imply the following result.

\begin{corollary}\label{metmini}
If $\mu$ is a minimal Hermitian metric, then $\RR^{*}\U(n)\cdot\mu$ parameterizes all minimal
Hermitian metrics on $(N_\mu, J)$.
\end{corollary}

\begin{example}\label{ejmmin}
For $t\in (0,1]$, consider the $2$-step nilpotent Lie algebra whose bracket is given by
$$
\begin{array}{lll}
\mu_{t}(e_1,e_2)=\sqrt{t}e_5, & \mu_{t}(e_1,e_4)=\frac{1}{\sqrt{t}}e_6,\\
\mu_{t}(e_2,e_3)=-\frac{1}{\sqrt{t}}e_6, & \mu_{t}(e_3,e_4)=-\sqrt{t}e_5.
\end{array}
$$
Let
{\small
$$
\begin{array}{lll}
J=\left[\begin{smallmatrix} 0&-1&&&&\\ 1&0&&&&\\ &&0&-1&&\\
&&1&0&&\\ &&&&0&-1\\ &&&&1&0
\end{smallmatrix}\right], && \la e_i,e_j\ra=\delta_{ij}.
\end{array}
$$}

A straightforward verification shows that $J$ is an abelian complex structure on $N_{\mu_{t}}$ for all
$t$, $\ip$ is compatible with $(N_{\mu_{t}},J)$, and the Ricci operator is given by
{\small
$$\Ric_{\mu_{t}}=\left[
    \begin{array}{cc}
      -\frac{1}{2}\left(\frac{t^2+1}{t}\right)I_{4} &  \\
       & \begin{matrix}
              t & 0 \\
              0 & 1/t \\
            \end{matrix}\\
    \end{array}
  \right].
$$}
By definition, we have
{\small
$$\Ricc_{\mu_{t}}=\left[
    \begin{array}{cc}
      -\left(\frac{t^2+1}{2t}\right)I_{4} &  \\
       & \left(\frac{t^2+1}{2t}\right)I_{2} \\
    \end{array}
  \right]=\frac{t^2+1}{2t}\left(-3I+2
\left[\begin{smallmatrix} 1&&&&&\\ &1&&&&\\ &&1&&&\\
&&&1&&\\ &&&&2&\\ &&&&&2
\end{smallmatrix}\right]
\right),
$$}
and thus $\mu_{t}$ is minimal of type $(1<2;4,2)$ by Theorem \ref{RS}.
It follows from
$$\Ric_{\mu_{t}}|_{\ngo_2}=
\left[
  \begin{array}{cc}
    t & 0 \\
    0 & 1/t \\
  \end{array}
\right],
$$
that the Hermitian nilmanifolds $\{(N_{\mu_{t}}, J, \ip) : 0 < t \leq 1\}$  are pairwise
non-isometric. \linebreak Indeed, if there exists $c\in\RR^{*}$ and $\varphi\in \U(3)\subset\Or(6)$ such that
$c\mu_s = \varphi\cdot\mu_t$ (see Corolario \ref{metmini}), then $\varphi = \left[\begin{smallmatrix}
\varphi_1&\\&\varphi_2\end{smallmatrix}\right]\in \U(2)\times\U(1)$ (recall that it is of type (4,2)) and
\linebreak $c^2\Ric_{\mu_{s}}|_{\ngo_2} = \varphi_2\Ric_{\mu_{t}}|_{\ngo_2}\varphi_{2}^{-1}$, hence
$c^2\left[\begin{smallmatrix}s&\\&1/s\end{smallmatrix}\right] =
\left[\begin{smallmatrix}t&\\&1/t\end{smallmatrix}\right]$. By taking quotients of their eigenvalues
we deduce that $s^2=t^2$ or $s^2=1/t^2$, which gives $s=t$ if
$s,t\in (0,1]$. We therefore obtain a curve $\{(N_{\mu_{t}}, J) : 0 < t \leq 1\}$ of pairwise
non-isomorphic abelian complex nilpotent Lie groups, by the uniqueness in result Theorem \ref{RS}
(see \cite{praga} for more examples).
\end{example}

From the above results, the problem of distinguishing two complex structures can be stated
as follows: if we fix the nilpotent Lie group $N$ then the $\Gl_{2n}(\RR)$-invariants
give us all possible complex structures on $N$ (Definition \ref{holomiso}), and the
$\Or(2n)$-invariants distinguish their respective minimal metrics (if any), up to scaling
(Corollary \ref{CRS}). If we now fix a $2n$-dimensional vector space and vary the brackets, the $\Gl_{n}(\CC)$-invariants provide the posible compatible metrics with a given complex structure
(see identification (\ref{metrcorch})), and the $\U(n)$-invariants their respective minimal metrics
(if any), up to scaling (see Corollary \ref{metmini}). In the latter case, the above example shows
how to use one of the Riemannian invariants: the eigenvalues of the Ricci operator. Since this
is not always possible, in the next section we will introduce a new invariant applicable
to 2-step nilpotent Lie algebras.

%-----------------------------------------------------------------------------------------

\subsection{Pfaffian form}\label{Pfaffian}
With the purpose to differentiate Lie algebras, up to isomorphism, we assign to each one a unique
homogeneous polynomial called  the \textit{Pfaffian form}, and by Proposition \ref{isoforms}
we will use the known polynomial invariants to obtain curves or families of brackets in a
vector space given. We follow the notation used in \cite{pfaff}.

Let $\ngo$ be a real Lie algebra, with Lie bracket $\mu$, and fix an inner product $\ip$ on $\ngo$.
For each $Z\in\ngo$ consider the skew-symmetric $\RR$-linear transformation
$J_Z:\ngo\longrightarrow\ngo$ defined by
\begin{equation}\label{jota}
\langle J_ZX,Y\rangle=\langle \mu(X,Y),Z\rangle, \qquad\forall\; X,Y\in\ngo.
\end{equation}

If $\ngo$ and $\ngo^\prime$ are two real Lie algebras and $J$, $J'$ are the
corresponding maps, relative to the inner products $\ip$ and $\ip'$ respectively,
then it is easy to see that a linear map $B:\ngo \to \ngo^\prime$ is a Lie
algebra isomorphism if and only if
\begin{equation}\label{propJZ}
B^t J_{Z}^\prime B = J_{B^t Z}, \quad \forall\; Z\in\ngo^\prime,
\end{equation}
where $B^t:\ngo^\prime \to \ngo$ is given by $\la B^t X, Y\ra = \la X, BY\ra^\prime$ for all
$X\in \ngo^\prime$, $Y\in \ngo$.

Assume now that $\ngo$ is $2$-step nilpotent and the decomposition  $\ngo=\ngo_1\oplus\ngo_2$
satisfies $\ngo_2=[\ngo,\ngo]$. \ If $\la \ngo_1, \ngo_2\ra=0$,
then $\ngo_1$ is $J_Z$-invariant for any $Z$ and $J_Z=0$ if and only if $Z\in \ngo_1$. Under
these conditions, the {\it Pfaffian form} $f:\ngo_2\to \RR$ of $\ngo$ is defined by
$$
f(Z)=\Pf(J_Z|_{\ngo_1}), \quad Z\in \ngo_2,
$$
where \ $\Pf:\sog(\ngo_1,\RR)\to \RR$ is the {\it Pfaffian}, that is, the only polynomial
function satisfying $\Pf(B)^2=\det{B}$ for all $B\in\sog(\ngo_1,\RR)$ and
\ $\Pf(J)=1$ for
\begin{align}\label{Jstand}
J=\left[\begin{smallmatrix} 0&-1&&&&\\ 1&0&&&&\\ &&0&-1&&\\
&&1&0&&\\ &&&&\ddots&&&\\ &&&&&0&-1\\&&&&&1&0
\end{smallmatrix}\right].
\end{align}
Note that we need $\dim{\ngo_1}$ to be even in order to get $f\ne 0$. \ Furthermore, if
$\dim{\ngo_1} = 2m$ and $\dim{\ngo_2} = k$ then the Pfaffian form $f=f(x_1,\ldots,x_k)$
of $\ngo$ is a homogeneous polynomial of degree $m$ in $k$ variables with coefficients in $\RR$.

Let $P_{k,m}(K)$ denote the set of all homogeneous polynomials of degree $m$ in $k$
variables with coefficients in a field $K$.

\begin{definition}\label{equiv} {\rm For $f,g\in P_{k,m}(K)$, we say that $f$ is
{\it projectively equivalent} to $g$, and denote it by $f\prek g$, if there exists
$A\in\Gl_k(K)$ and $c\in K^*$ such that}
$$
f(x_1,...,x_k)=cg(A(x_1,...,x_k)).
$$
\end{definition}

\begin{remark}\label{prequiv}
If $f, g \in P_{k,m}(\RR)$, then
\[ f\preq g \Leftrightarrow
\begin{cases}
f\in\Gl_k(\RR)\cdot g, & \text{if $m$ is odd},\\
f\in\ \pm \Gl_k(\RR)\cdot g, & \text{if $m$ is even}.
\end{cases} \]
Recall that
$(A\cdot f)(x_1,\ldots, x_k) = f(A^{-1}(x_1,\ldots, x_k))$ for all $A\in\Gl_k(K)$,
$f\in P_{k,m}(K)$.
\end{remark}

\begin{proposition}\cite{pfaff}\label{isoforms}
Let $\ngo,\ngo'$ be two-step nilpotent Lie algebras over $\RR$.  If $\ngo$
and $\ngo'$ are isomorphic then $f\preq f'$, where $f$ and $f'$ are the Pfaffian
forms of $\ngo$ and $\ngo'$, respectively.
\end{proposition}

The above proposition says that the projective equivalence class of the form
$f(x_1,\ldots,x_k)$ is an isomorphism invariant of the Lie algebra $\ngo$. Note that
if we do the composition $I\circ f(\mu)$ of the Pfaffian form $f(\mu)$ with an invariant
$I\in P_{k,m}(\RR)^{\Sl_k(\RR)}$ (the ring of invariant \linebreak polynomials), we obtain scalar
$\Sl_k(\RR)$-invariants. Moreover, if we consider quotients of same degree of the form
$\frac{I_1(f(\mu))}{I_2(f(\mu))}$  we obtain $\Gl_k(\RR)$-invariants (see Example \ref{ejmfaff}).

In what follows, we give some basic properties of the Pfaffian form and some invariants
for binary quartic forms.
\begin{itemize}
\item[(i)] If $A$ is a skew symmetric matrix of order $4\times4$, say
$$
A=\left[\begin{smallmatrix} 0&b_{12}&b_{13}&b_{14}\\
-b_{12}&0&b_{23}&b_{24}\\
-b_{13}&-b_{23}&0&b_{34}\\
-b_{14}&-b_{24}&-b_{34}&0
\end{smallmatrix}\right],$$
then $\Pf(A)=b_{12}b_{34}-b_{13}b_{24}+b_{14}b_{23}$.
\item[(ii)] $\Pf\left(\left[\begin{smallmatrix}
A_1& 0\\
0 & A_2
\end{smallmatrix}\right]\right) = \Pf(A_1)\Pf(A_2)$.
\item[(iii)] Let $p(x,y) = \sum_{i=0}^{4}a_{i}x^{4-i}y^{i} \in P_{2,4}(\RR)$. Define
\begin{align*}
& S(p):= a_0 a_4 - 4a_1 a_3 + 3a_2^2.\\
& T(p):= a_0 a_2 a_4 - a_0 a_3^2 + 2a_1 a_2 a_3 - a_1^2 a_4 - a_2^3.
\end{align*}
We have that $S$ and $T$ are $\Sl_2(\RR)$-invariant (see for instance \cite{invariant}), that is \linebreak
$S(g\cdot p) = S(p)$ and $T(g\cdot p)= T(p)$ for any $p\in P_{2,4}(\RR)$, $g\in\Sl_2(\RR)$. \linebreak
Moreover, $S(c p) = c^2 S(p)$ and $T(c p)= c^3 T(p)$ for all $c\in \RR$.
\end{itemize}

\begin{example}\label{ejmfaff}
Let $\ngo$ be the $2$-step nilpotent Lie algebra whose bracket is defined, for any $t\in \RR$, by
\begin{align*}
& \lambda_t(X_1, X_3) = - \lambda_t(X_2, X_4) = Z_1,\\
& \lambda_t(X_1, X_4) = \lambda_t(X_2, X_3) = \lambda_t(X_5, X_8) = \lambda_t(X_6, X_7) = - Z_2,\\
& \lambda_t(X_5, X_7) = - \lambda_t(X_6, X_8) = tZ_1.
\end{align*}
Consider the inner product $\la X_i,X_j\ra= \la Z_i,Z_j\ra=\delta_{ij}$. In this case
$\ngo_1=\la X_1,...,X_8\ra_{\RR}$ and $\ngo_2=\la Z_1,Z_2\ra_{\RR}$.
If $Z = x Z_1 + y Z_2$, with $x,y\in\RR$, then
$$
J_Z|_{\ngo_1}=\left[\begin{smallmatrix}
& & -x & y & & & &\\
& & y & x & & & &\\
x & -y &  &  & & & &\\
-y & -x & & & & & &\\
& & & & & & -tx & y\\
& & & & & & y & tx\\
& & & & tx & -y & &\\
& & & & -y & -tx & &\\
\end{smallmatrix}\right].$$
By definition (see also properties (i) and (ii) above), the Pfaffian form of $\ngo$ is
$$f_t := f(\lambda_t) = (x^2 + y^2)(t^2x^2 + y^2)=
t^2 x^4 + (t^2+1) x^2 y^2 + y^4.$$
We claim that if $f_t\preq f_s$ then $t=s$ for all $t,s$ in any of the following intervals:
$$(-\infty,-1], [-1,0], [0,1], [1,\infty).$$ Indeed, by assumption, there exists $c\in\RR^*$
and $g\in \Gl_2(\RR)$ such that $c \ g\cdot f_s = f_t$. From this we deduce that there exists
$\widetilde{c}\in\RR^*$ and $\widetilde{g}\in \Sl_2(\RR)$ such that
$\widetilde{c} \ \widetilde{g}\cdot f_s = f_t$. For all $t\in\RR$, define the function (see (iii) above)
$$h(t):= \frac{S(f_t)^3}{T(f_t)^2}.$$
It follows that
$$h(t) = \frac{S(f_t)^3}{T(f_t)^2} =
\frac{S(\widetilde{c} \ \widetilde{g}\cdot f_s)^3}{T(\widetilde{c} \ \widetilde{g}\cdot f_s)^2} =
\frac{\widetilde{c}^6 \ S(\widetilde{g}\cdot f_s)^3}{\widetilde{c}^6 \ T(\widetilde{g}\cdot f_s)^2} =
\frac{S(f_s)^3}{T(f_s)^2} = h(s).$$
It follows that
$$h(t)= \frac{(3t^4+7t^2+3)^3}{(t^2+1)^2(t^2+t+1)^2(t^2-t+1)^2}$$
Since the derivative of $h(t)$ only vanishes at $-1, 0, 1$, we conclude that $h$ is injective on
any of the intervals mentioned above. Proposition \ref{isoforms} now shows that $\{(\ngo, \lambda_t): t\in[1,\infty)\}$
(or $t$ in any of the other intervals) is a pairwise non-isomorphic family of Lie algebras.
\end{example}

If we take $\Gl_n(\CC):=\{g\in\Gl_{2n}(\RR) : gJ = Jg\}$, where $J$ is
given by (\ref{Jstand}), we can state the analogue of Proposition \ref{isoforms},
which will be crucial in Section \ref{result6}.

\begin{proposition}\label{PfafCC}
Suppose that $\ngo=\ngo_1\oplus\ngo_2$, with $\dim{\ngo_1} = 2p$ and $\dim{\ngo_2} = 2q$,
and $J\ngo_{i}=\ngo_{i}$. Assume $\mu, \lambda\in\Lambda^2\ngo_1^*\otimes\ngo_2$ satisfy
$\mu(\ngo_1,\ngo_1)=\lambda(\ngo_1,\ngo_1)=\ngo_2$. If $\lambda\in \Gl_{n}(\CC)\cdot \mu$ (n=p+q),
then $$f(\lambda)\in \RR_{>0}\Gl_{q}(\CC)\cdot f(\mu),$$
where $f(\mu)$, $f(\lambda)$ are the Pfaffian forms of $(\ngo, \mu)$ and $(\ngo, \lambda)$,
respectively.
\end{proposition}
\begin{proof}
Let $\hg:=(\ngo,\mu)$, $\hg^\prime:=(\ngo,\lambda)$ and $J_\mu$, $J_\lambda$ the corresponding
maps, relative to the inner products on $\ngo$ (see (\ref{jota})). Suppose that
$g\cdot\mu = \lambda$ with $g\in\Gl_n(\CC)$ (i.e. $g\in\Gl_{2n}(\RR)$, $gJ=Jg$).
By assumption, $g = \left[\begin{smallmatrix}g_1&\\&g_2\end{smallmatrix}\right]\in \Gl_p(\CC)\times\Gl_q(\CC)$
and $g:\hg \to \hg^\prime$ is a Lie algebra isomorphism satisfying
\ $g\ngo_1=\ngo_1$ and $g\ngo_2=\ngo_2$. It follows from (\ref{propJZ}) that
$$g^t J_{\lambda}(Z) g = J_{\mu}(g^t Z), \quad \forall Z\in\ngo_1,$$
and since the subspaces $\ngo_1$ and $\ngo_2$ are preserved by $g$ y $g^t$ we have that
$$f^\prime(Z)=cf(g_{2}^t Z),$$
where $c^{-1}=\det g_1 > 0$ ($\Gl_p(\CC)$ is connected) and
$g_{2}^t:\lambda(\ngo_1,\ngo_1) \to \mu(\ngo_1,\ngo_1)$. It is clear that $g_2^t\in \Gl_{2q}(\RR)$ and
satisfies
$$\la J g_{2}^t Z, Y\ra=\la g_{2}^t Z, -JY\ra=\la Z, g_2(-JY)\ra=\la Z, -J g_2 Y\ra=
\la J Z, g_2Y\ra=\la g_{2}^t J Z, Y\ra.$$
Thus $g_{2}^t\in\Gl_q(\CC)$ and we conclude that $f(\lambda)\in \RR_{>0}\Gl_{q}(\CC)\cdot f(\mu)$.
\end{proof}

We end this section with an example of two homogeneous polynomials that are projectively equivalent
over $\RR$ but not over $\CC$ (in the sense of Proposition \ref{PfafCC}).

\begin{example}
In $\hg_5\times \RR$, define the Lie brackets $\mu^{+}$ and $\mu^{-}$ by
$$\mu^{\pm}(e_1,e_2)=e_6, \quad \mu^{\pm}(e_3,e_4)=\pm e_6.$$
Consider the inner product $\la e_i, e_j\ra= \delta_{ij}$. If $Z=x e_6$, with $x\in\RR$, then
\begin{align*}
\begin{array}{lll}
J_Z^{+}|_{\ngo_1}=\left[\begin{smallmatrix}
0&-x&&\\
x&0&&\\
&&0&-x\\
&&x&0
\end{smallmatrix}\right], &&
J_Z^{-}|_{\ngo_1}=\left[\begin{smallmatrix}
0&-x&&\\
x&0&&\\
&&0&x\\
&&-x&0
\end{smallmatrix}\right].
\end{array}
\end{align*}
Hence $f(\mu^{+})=x^2$ and $f(\mu^{-})=-x^2$. It follows that $f(\mu^{-})\preq f(\mu^{+})$ but
$$f(\mu^{-})\notin \RR_{>0}\U(1)\cdot f(\mu^{+}).$$ Recall that $\Gl_1(\CC)=\RR_{>0}\U(1)$.
\end{example}

%===============================================================================================

\section{Minimal metrics on $6$-dimensional abelian complex nilmanifolds}\label{result6}

The classification of 6-dimensional nilpotent real Lie algebras admitting a complex
structure was given in \cite{strcomplx}, and the abelian case in \cite{abelstruc}. Lately,
A. Andrada, M.L. Barberis and I.G. Dotti in \cite{AndBrbDtt} gave a classification of all Lie
algebras admitting an abelian complex structure; furthermore, they give a parametrization, on each
Lie algebra, of the space of abelian structures up to holomorphic isomorphism. In particular,
there are three nilpotent Lie algebras carrying curves of non-equivalent structures. \ Based on
this parametrization, we study the existence of minimal metrics on each of these complex nilmanifolds
(see Theorem \ref{result4}), and provide an alternative proof of the
pairwise non-isomorphism between the structures.

The classification in \cite{AndBrbDtt} fix the Lie algebra and varies the complex structure.
For example, on the Lie algebra $\hg_3\times\hg_3$ they found the curve $J_s$ of abelian complex
structures defined by $J_s e_1 = e_2, J_s e_3 = e_4, J_s e_5 = s e_5 + e_6$, $s\in\RR$, and fix
the bracket $[e_1, e_2] = e_5$, $[e_3, e_4] = e_6$. We now fix the complex structure and varies
the bracket as follows.

For $\ngo=\vg_1\oplus\vg_2$, with $\vg_1=\RR^4$ and $\vg_2=\RR^2$, \ consider the vector space
$\Lambda^2\vg_1^*\otimes\vg_2$ of all skew symmetric bilinear maps $\mu:\vg_1\times\vg_1\to\vg_2$.
Any $6$-dimensional $2$-step nilpotent Lie algebra with $\dim{\mu(\ngo,\ngo)}\leq 2$ can be modelled
in this way. \ Fix a basis of $\ngo$, say $\{ e_1,\ldots,e_6\}$, such that
$\vg_1=\la e_1,...,e_4\ra_{\RR}$, $\vg_2=\la e_5,e_6\ra_{\RR}$. The complex structure and the
compatible metric will be always defined by
\begin{align}\label{Jstand2}
\begin{array}{lll}
J:=\left[\begin{smallmatrix} 0&-1&&&&\\ 1&0&&&&\\ &&0&-1&&\\
&&1&0&&\\ &&&&0&-1\\ &&&&1&0
\end{smallmatrix}\right], && \la e_i,e_j\ra:=\delta_{ij}.
\end{array}
\end{align}

\begin{proposition}\label{holomiso2}
Let $(N_{\widetilde{\mu}}, \widetilde{J})$ be a complex nilmanifold, with
$\widetilde{\mu}\in\Lambda^2\vg_1^*\otimes\vg_2$. If there exists $g\in\Gl_{6}(\RR)$ such that
$g\widetilde{J}g^{-1} = J$, then $(N_{\widetilde{\mu}}, \widetilde{J})$ and $(N_{g\cdot\widetilde{\mu}}, J)$
are holomorphically isomorphic.
\end{proposition}

Returning to the above example, by choosing
\begin{align*}
g=\left[\begin{smallmatrix} 1&0&&&&\\ 0&1&&&&\\ &&1&0&&\\
&&0&1&&\\ &&&&1&-s\\ &&&&0&1
\end{smallmatrix}\right],
\end{align*}
we have $g J_{s} g^{-1} = J$, and therefore $(N_{[\cdot, \cdot]}, J_{s})$ and $(N_{\mu_3}, J)$
are holomorphically isomorphic by Proposition \ref{holomiso2}, where now the bracket is given
by $\mu_3(e_1, e_2) = e_5$ and $\mu_3(e_3, e_4) = -s e_5 + e_6$ with $s\in \RR$.
By arguing as above for each item in \cite[Theorem 3.5.]{AndBrbDtt}, we have obtained
Table \ref{tablecorch}.

\begin{table}
\centering
\renewcommand{\arraystretch}{1.6}
\begin{tabular}{|c|c|}\hline
$\mathbf{\ngo}$ & \textbf{Bracket} \\ \hline\hline
$\ngo_1:=\mathfrak{h}_3\times \RR^3$ & $\mu_1(e_1,e_2)=e_6$ \\ \hline
$\ngo_2:=\mathfrak{h}_5\times \RR$ & $\mu_2^{\pm}(e_1,e_2)=e_6$, \ $\mu_2^{\pm}(e_3,e_4)=\pm e_6$\\ \hline
$\ngo_3:=\mathfrak{h}_3\times \mathfrak{h}_3$ &
$\mu_3^{s}(e_1,e_2)=e_5$, \ $\mu_3^{s}(e_3,e_4)= - se_5 + e_6$ \\
& $s\in\RR$ \\ \hline
$\ngo_4:=\mathfrak{h}_3(\CC)$ &
$\mu^{t}_4(e_1,e_2)=\sqrt{t}e_5$, \ $\mu^{t}_4(e_1,e_4)=\frac{1}{\sqrt{t}}e_6$ \\ %\hline
 & $\mu^{t}_4(e_2,e_3)= - \frac{1}{\sqrt{t}}e_6$, \ $\mu^{t}_4(e_3,e_4)= - \sqrt{t}e_5$ \\
& $t\in (0,1]$ \\ \hline
$\ngo_5$ & $\mu_5(e_1,e_2)= e_5$, \ $\mu_5(e_1,e_4)= - e_6$ \\
& $\mu_5(e_2,e_3)= e_6$ \\ \hline
$\ngo_6$ & $\mu_6(e_1,e_2)= - e_3$, \ $\mu_6(e_1,e_4)= - e_6$ \\
& $\mu_6(e_2,e_3)= e_6$ \\ \hline
$\ngo_7$ & $\mu_7^{t}(e_1,e_2)= - e_4$, \ $\mu_7^{t}(e_1,e_3)=\sqrt{t}e_5$ \\
& $\mu_7^{t}(e_2,e_4)= \sqrt{t}e_5$, \ $\mu_7^{t}(e_1,e_4)= - \frac{1}{\sqrt{t}}e_6$ \\
& $\mu_7^{t}(e_2,e_3)= \frac{1}{\sqrt{t}}e_6$, \ $t\in (0,1]$ \\ \cline{2-2}
& $\widetilde{\mu}_7^{t}(e_1,e_2)= - e_4$, \ $\widetilde{\mu}_7^{t}(e_1,e_3)=
 \sqrt{-t}e_5$ \\
& $\widetilde{\mu}_7^{t}(e_2,e_4)= \sqrt{-t}e_5$, \ $\widetilde{\mu}_7^{t}(e_1,e_4)= \frac{1}{\sqrt{-t}}e_6$ \\
& $\widetilde{\mu}_7^{t}(e_2,e_3)= - \frac{1}{\sqrt{-t}}e_6$, \ $t\in [-1,0)$ \\ \hline
\end{tabular}
\vspace{0.3cm}
\caption{Abelian complex nilmanifolds of dimension $6$.}\label{tablecorch}
\end{table}
\begin{remark}
In the classification given in \cite{AndBrbDtt}, they incorrectly claim that the curves of
structures $J_t^1$ and $J_t^2$ on $\ngo_4$ are non-equivalent (see a corrected version at arXiv:0908.3213).
Indeed, the matrix $g$ defined in (\ref{matriz}) is an automorphism of $\ngo_4$ and $gJ_{t}^{1}g^{-1} = J_t^2$,
hence $J_t^1$ and $J_t^2$ are equivalent. Note that in
Table \ref{tablecorch} only appears a `curve' (it is proved below) of brackets on $\ngo_4$, which is due
to the following proposition and Theorem \ref{result4}. The brackets $\mu^{1,t}_4$ and $\mu^{2,t}_4$
are obtained  from the curves of structures $J_t^1$ and $J_t^2$, respectively.
\end{remark}

\begin{proposition}
\ $\mu^{2,t}_4 \in \U(2)\times\U(1)\cdot \mu^{1,t}_4$ for all $t\in (0,1]$, where the brackets
$\mu^{1,t}_4, \mu^{2,t}_4$ on $\ngo_4$ are given by
{\small
\begin{align*}
&\mu^{1,t}_4(e_1,e_2)=\sqrt{t}e_5, \ \ \mu^{1,t}_4(e_1,e_4)=\frac{1}{\sqrt{t}}e_6, \ & &
\mu^{2,t}_4(e_1,e_3)=\sqrt{t}e_5, \ \ \mu^{2,t}_4(e_2,e_4)=\sqrt{t}e_5,\\
&\mu^{1,t}_4(e_2,e_3)= - \frac{1}{\sqrt{t}}e_6, \ \ \mu^{1,t}_4(e_3,e_4)= - \sqrt{t}e_5. & &
\mu^{2,t}_4(e_1,e_4)= - \frac{1}{\sqrt{t}}e_6, \ \ \mu^{2,t}_4(e_2,e_3)=\frac{1}{\sqrt{t}}e_6.
\end{align*}}
\end{proposition}
\begin{proof}
We have
$$
\begin{array}{lll}
g=\begin{bmatrix}
\frac{\sqrt{2}}{2}i&-\frac{\sqrt{2}}{2}&0\\
&&\\
\frac{\sqrt{2}}{2}&-\frac{\sqrt{2}}{2}i&0\\
&&\\
0&0&1
\end{bmatrix} & \in \ \U(2)\times\U(1).
\end{array}
$$
Using the identification
$a+bi\mapsto\left[\begin{smallmatrix} a&-b\\b&a \end{smallmatrix}\right]$, we thus get
\begin{equation}\label{matriz}
g=\left[\begin{smallmatrix}
0&-\frac{\sqrt{2}}{2}&-\frac{\sqrt{2}}{2}&0&0&0\\
&&&&&\\
\frac{\sqrt{2}}{2}&0&0&-\frac{\sqrt{2}}{2}&0&0\\
&&&&&\\
\frac{\sqrt{2}}{2}&0&0&\frac{\sqrt{2}}{2}&0&0\\
&&&&&\\
0&\frac{\sqrt{2}}{2}&-\frac{\sqrt{2}}{2}&0&0&0\\
&&&&&\\
0&0&0&0&1&0\\
&&&&&\\
0&0&0&0&0&1
\end{smallmatrix}\right]
\end{equation}
By definition, it follows that
\begin{itemize}{\small
\item $\mu^{2,t}_4(e_1,e_2)=0.$\\
$g\cdot\mu^{1,t}_4(e_1,e_2)=g\mu^{1,t}_4\left(-\frac{\sqrt{2}}{2}e_2-\frac{\sqrt{2}}{2}e_3,
\frac{\sqrt{2}}{2}e_1-\frac{\sqrt{2}}{2}e_4\right)=
g\{\frac{1}{2}\left(\sqrt{t}e_5-\sqrt{t}e_5\right)\}=0.$

\item $\mu^{2,t}_4(e_1,e_3)=\sqrt{t}e_5.$\\
$g\cdot\mu^{1,t}_4(e_1,e_3)=g\mu^{1,t}_4\left(-\frac{\sqrt{2}}{2}e_2-\frac{\sqrt{2}}{2}e_3,
\frac{\sqrt{2}}{2}e_1+\frac{\sqrt{2}}{2}e_4\right)=
g\{\frac{1}{2}\left(\sqrt{t}e_5+\sqrt{t}e_5\right)\}=\sqrt{t}e_5.$

\item $\mu^{2,t}_4(e_1,e_4)=-\frac{1}{\sqrt{t}}e_6.$\\
$g\cdot\mu^{1,t}_4(e_1,e_4)=g\mu^{1,t}_4\left(-\frac{\sqrt{2}}{2}e_2-\frac{\sqrt{2}}{2}e_3,
\frac{\sqrt{2}}{2}e_2-\frac{\sqrt{2}}{2}e_3\right)=
g\{\frac{1}{2}\left(-\frac{1}{\sqrt{t}}e_6-\frac{1}{\sqrt{t}}e_6\right)\}=-\frac{1}{\sqrt{t}}e_6.$

\item $\mu^{2,t}_4(e_2,e_3)=\frac{1}{\sqrt{t}}e_6.$\\
$g\cdot\mu^{1,t}_4(e_2,e_3)=g\mu^{1,t}_4\left(\frac{\sqrt{2}}{2}e_1-\frac{\sqrt{2}}{2}e_4,
\frac{\sqrt{2}}{2}e_1+\frac{\sqrt{2}}{2}e_4\right)=
g\{\frac{1}{2}\left(\frac{1}{\sqrt{t}}e_6+\frac{1}{\sqrt{t}}e_6\right)\}=\frac{1}{\sqrt{t}}e_6.$

\item $\mu^{2,t}_4(e_2,e_4)=\sqrt{t}e_5.$\\
$g\cdot\mu^{1,t}_4(e_2,e_4)=g\mu^{1,t}_4\left(\frac{\sqrt{2}}{2}e_1-\frac{\sqrt{2}}{2}e_4,
\frac{\sqrt{2}}{2}e_2-\frac{\sqrt{2}}{2}e_3\right)=
g\{\frac{1}{2}\left(\sqrt{t}e_5+\sqrt{t}e_5\right)\}=\sqrt{t}e_5.$

\item $\mu^{2,t}_4(e_3,e_4)=0.$\\
$g\cdot\mu^{1,t}_4(e_3,e_4)=g\mu^{1,t}_4\left(\frac{\sqrt{2}}{2}e_1+\frac{\sqrt{2}}{2}e_4,
\frac{\sqrt{2}}{2}e_2-\frac{\sqrt{2}}{2}e_3\right)=
g\{\frac{1}{2}\left(\sqrt{t}e_5-\sqrt{t}e_5\right)\}=0.$}
\end{itemize}
Hence \ $g\cdot\mu^{1,t}_4 = \mu^{2,t}_4$, \ which completes the proof.
\end{proof}

\begin{table}
\centering
\renewcommand{\arraystretch}{1.6}
\begin{tabular}{|c|c|c|c|}\hline
$\mathbf{\ngo}$ & \textbf{Bracket} & \textbf{Type} & \textbf{Minimal}\\ \hline\hline
$\ngo_1$ & $\mu_1(e_1,e_2)=e_6$ & {\small $(3<5<6;2,2,2)$} & Yes \\ \hline
$\ngo_2$ & $\mu_2^{\pm}(e_1,e_2)=e_6$, \ $\mu_2^{\pm}(e_3,e_4)=\pm e_6$ &
$(1<2;4,2)$ & Yes \\ \hline
$\ngo_3$ &
$\mu^{s}_3(e_1,e_2)=e_5$, \ $\mu^{s}_3(e_3,e_4)= \frac{-s}{\sqrt{1+s^2}}e_5 + \frac{1}{\sqrt{1+s^2}}e_6$ & $(1<2;4,2)$ & Yes \\ & $s\in\RR$ & & \\ \hline
$\ngo_4$ &
$\mu^{t}_4(e_1,e_2)=\sqrt{t}e_5$, \ $\mu^{t}_4(e_1,e_4)=\frac{1}{\sqrt{t}}e_6$ & $(1<2;4,2)$ & Yes \\ %\hline
 & $\mu^{t}_4(e_2,e_3)= - \frac{1}{\sqrt{t}}e_6$, \ $\mu^{t}_4(e_3,e_4)= - \sqrt{t}e_5$ &  &  \\
& $t\in (0,1]$ & & \\ \hline
$\ngo_5$ & $\mu_5(e_1,e_2)= e_5$, \ $\mu_5(e_1,e_4)= - e_6$ & ------ & No \\
& $\mu_5(e_2,e_3)= e_6$ & & \\ \hline
$\ngo_6$ & $\mu_6(e_1,e_2)= - e_3$, \ $\mu_6(e_1,e_4)= - e_6$ & {\small $(1<2<3;2,2,2)$} & Yes \\
& $\mu_6(e_2,e_3)= e_6$ & & \\ \hline
$\ngo_7$ & $\mu_7^{t}(e_1,e_2)= - \sqrt{t+1/t}e_4$, \ $\mu_7^{t}(e_1,e_3)=\sqrt{t}e_5$ & {\small $(1<2<3;2,2,2)$} &
Yes \\
& $\mu_7^{t}(e_2,e_4)= \sqrt{t}e_5$, \ $\mu_7^{t}(e_1,e_4)= - \frac{1}{\sqrt{t}}e_6$ & & \\
& $\mu_7^{t}(e_2,e_3)= \frac{1}{\sqrt{t}}e_6$, \ $t\in (0,1]$ & & \\ \cline{2-2}
& \mbox{\small$\widetilde{\mu}_7^{t}(e_1,e_2)= - \sqrt{-t-1/t}e_4$}, \ $\widetilde{\mu}_7^{t}(e_1,e_3)=
 \sqrt{-t}e_5$ & & \\
& $\widetilde{\mu}_7^{t}(e_2,e_4)= \sqrt{-t}e_5$, \ $\widetilde{\mu}_7^{t}(e_1,e_4)= \frac{1}{\sqrt{-t}}e_6$ & & \\
& $\widetilde{\mu}_7^{t}(e_2,e_3)= - \frac{1}{\sqrt{-t}}e_6$, \ $t\in [-1,0)$ & & \\ \hline
\end{tabular}
\vspace{0.3cm}
\caption{Minimal metrics on $6$-dimensional abelian complex nilmanifolds.}\label{tablecurv}
\end{table}

\begin{theorem}\label{result4}
Any $6$-dimensional abelian complex nilmanifold admits a minimal metric, with the only
exception of $(N_5, J)$.
\end{theorem}

\begin{proof}
By applying Theorem \ref{RS} (as we described in Example \ref{ejmmin} for $\ngo_4$),
it is easily seen that $(N_1, J)$ admit a minimal metric of type $(3<5<6;2,2,2)$; $(N_2, J)$,
$(N_3, J)$ and $(N_4, J)$ one of type $(1<2;4,2)$; $(N_6, J)$ and $(N_7, J)$ one of type
$(1<2<3;2,2,2)$. Furthermore, we can see that each $\mu_i$ on $\ngo_i$ is minimal, if $i\neq 5$
(column 4, Table \ref{tablecurv}). Note that the Table \ref{tablecurv} differs from the
Table \ref{tablecorch} in $\ngo_3$ and $\ngo_7$, this is due to get $\mu_3$ and $\mu_7$ minimals
was required to act with a matrix $g\in\Gl_3(\CC)$ in the brackets given in the Table \ref{tablecorch}.
For example, for $\ngo_7$, take
\begin{align*}
g=\left[\begin{smallmatrix} \alpha&&\\ &\frac{1}{\alpha}&\\ &&1
\end{smallmatrix}\right],
\end{align*}
where $\alpha=(t+\frac{1}{t})^{-\frac{1}{6}}$ for $\mu_7^{t}$, and
$\alpha=(-t-\frac{1}{t})^{-\frac{1}{6}}$ for $\widetilde{\mu}_7^{t}$.

It remains to prove that $(N_5, J)$ does not admit minimal
compatible metrics. To do this, we will use some properties of the $\Gl_n(\RR)$-invariant
stratification for the representation \linebreak $\Lambda^2(\RR^n)^*\otimes\RR^n$ of $\Gl_n(\RR)$
(see \cite{standard}, \cite{einstein} for more details).

Let \ {\small$\beta = diag(-1/2,-1/2,-1/2,-1/2,1/2,1/2)$}. Hence
$$G_{\beta}:= \left\{g\in \Gl(6) : g\beta g^{-1} = \beta, \ g J g^{-1} = J\right\}=
\Gl_2(\CC)\times \Gl_1(\CC).$$
Since $\ggo_{\beta}=\RR\beta \oplus^{\bot} \hg_{\beta}$, it follows that $\hg_{\beta}$
is Lie subalgebra. Let $H_{\beta}\subset G_{\beta}$ denote the Lie subgroup with Lie
algebra $\hg_{\beta}$.
We thus get
$$\hg_{\beta}= \left\{\left[
                                \begin{array}{cc}
                                  A & 0 \\
                                  0 & B \\
                                \end{array}
                              \right]:
                                trA = trB
\right\}, \quad H_{\beta}= \left\{\left[
                                \begin{array}{cc}
                                  g & 0 \\
                                  0 & h \\
                                \end{array}
                              \right]:
                                det(g) = det(h)
\right\}.$$
But $\mathfrak{h}_{\beta}= \left(\RR\left[\begin{smallmatrix} I&\\ & 2I
\end{smallmatrix}\right]\right)\oplus \widetilde{\mathfrak{h}}_{\beta}$ where
$$\widetilde{\mathfrak{h}}_{\beta}= \left\{\left[
                                \begin{array}{cc}
                                  A & 0 \\
                                  0 & B\\
                                \end{array}
                              \right]:
                                trA = trB =0
\right\}.$$
This clearly forces $\widetilde{H}_{\beta}= \Sl_2(\CC) \times \{I\}$. Therefore, it suffices to
prove that \ $0\notin \overline{\Sl_2(\CC)\cdot \mu_5}$ \ and \
$\mu_2\in \overline{\Sl_2(\CC)\cdot \mu_5}$, with $\mu_2$ and $\mu_5$ the brackets of  $\mathfrak{n}_2$
and $\mathfrak{n}_5$ respectively, which is due to the fact that $G\cdot \mu$ is minimal if and only
if $H_{\beta}\cdot \mu$ is closed (see for instance \cite[Theorem 9.1.]{einstein}). Indeed,
an easy computation shows that
$$\left[\begin{smallmatrix} a&&&\\ &a&&\\&&1/a&\\&&&1/a
\end{smallmatrix}\right]\cdot \mu_5 \longrightarrow \mu_2 \quad {\rm letting} \ a\to\infty.$$
From what has already been and the fact that $\Sl_2(\CC)\cdot \mu_2$ is closed ($\mathfrak{n}_2$ is
minimal), we conclude that  $0\notin \overline{\Sl_2(\CC)\cdot \mu_5}$ \ by the uniqueness of closed
orbits in the closure of an orbit (note that  $\{0\}$ is a closed orbit).
\end{proof}

We now will use the Pfaffian forms to give an alternative proof of the pairwise non-isomorphism of
the family given in \cite[Theorem 3.5.]{AndBrbDtt} in the $2$-step nilpotent case. Since
$\dim{\vg_1} = 4$ and $\dim{\vg_2} = 2$ , the Pfaffian forms of
$\ngo_1,\ldots,\ngo_5$ belong to the set $P_{2,2}(\RR)$; so we are left with the task of
determining the quotient $P_{2,2}(\RR)/\Gl_1(\CC) = P_{2,2}(\RR)/\RR_{>0}\U(1)$ (see Proposition \ref{PfafCC}).

Using the identification
$P = ax^2+bxy+cy^2 \leftrightarrow P_{\scriptscriptstyle A}:= \la A (x,y), (x,y) \ra$,
where $A = \left[\begin{smallmatrix}a & b/2\\ b/2 & c\end{smallmatrix}\right]$, we have
(see Remark \ref{prequiv})
\begin{equation*}
\begin{aligned}
P_{2,2}(\RR) / \pm \Gl_2(\RR) \ = \
\end{aligned}
\left\{
\begin{aligned}
&x^2+y^2,\\
&x^2-y^2,\\
&x^2,\\
& 0.
\end{aligned}
\right.
\end{equation*}
Proposition \ref{PfafCC} now implies that
$$P_{2,2}(\RR)/\RR_{>0}\U(1) = \{ax^2+by^2 : a\leq b, a^2+b^2=1\} \cup \{0\}.$$
This allows us to classify the Pfaffian forms of $\ngo_1,\ldots,\ngo_5$,
which is summarized in Figure \ref{grpfaffian}. The Lie algebra $\ngo_4^*$ is given by
$\mu_t(e_1,e_3)=-ts e_6$, $\mu_t(e_1,e_4)= \mu_t(e_2,e_3)= s e_5$,
$\mu_t(e_2,e_4)= s(2-t) e_6$, with $s=\sqrt{2+t^2+(2-t)^2}, \ 1\leq t < 2$;
it is minimal and $(N_{\mu_t}, J)$ is not abelian (see \cite[Example 5.3.]{canonical}).

From Figure 1, it is clear that $\ngo_3$ and $\ngo_4$ have (minimal) Hermitian metric curves;
$(\ngo_2, \mu_2^{+})$ and $(\ngo_2, \mu_2^{-})$ are distinguished; $\ngo_1$ has an unique
(minimal) Hermitian metric; and $\ngo_5$ has an unique Hermitian metric.

We now consider the Lie algebras which are not
$2$-step nilpotent. The Lie algebra $\ngo_6$ has an unique minimal metric up to isometry and scaling,
by Theorem \ref{RS}. For $\ngo_7$, an easy computation shows that for all $t\in [-1,0), s\in(0,1]$
$$\Ric_{\widetilde{\mu}_7^{t}}\mid_{\zg}=
\left[
  \begin{array}{cc}
    -t & 0 \\
    0 & -1/t \\
  \end{array}
\right],\qquad
\Ric_{\mu_7^{s}}\mid_{\zg}=
\left[
  \begin{array}{cc}
    s & 0 \\
    0 & 1/s \\
  \end{array}
\right],
$$
where $\zg:=\la e_5, e_6\ra_{\RR}$. From this we deduce that the Hermitian nilmanifolds
$\{(N_{\mu_7^t}, J, \ip) : t\in (0,1]\}$ are pairwise non-isometric
(as we described in Example \ref{ejmmin} for $\ngo_4$). Likewise for
$\{(N_{\mbox{\small$\widetilde{\mu}_7^{t}$}}, J, \ip) : t\in [-1,0)\}$.

We will distinguish $\mu_7^t$, $t\in (0,1]$, of \ $\mbox{\small$\widetilde{\mu}_7^{t}$}$, $t\in [-1,0)$. To do this we need the following (see (\ref{prodtint}))
\begin{align*}
& \|\mu_7^t\|^2=2\left(\|\mu_7^{t}(e_1, e_2)\|^2+
\|\mu_7^{t}(e_1, e_3)\|^2+\|\mu_7^{t}(e_1, e_4)\|^2+
\|\mu_7^{t}(e_2, e_3)\|^2+
\|\mu_7^{t}(e_2, e_4)\|^2\right)\\
&\hspace{1.1cm}= 6\left(t+\frac{1}{t}\right), \ \ t\in (0,1].\\
& \|\mbox{\small$\widetilde{\mu}_7^{t}$}\|^2=
2\left(\|\mbox{\small$\widetilde{\mu}_7^{t}$}(e_1, e_2)\|^2+
\|\mbox{\small$\widetilde{\mu}_7^{t}$}(e_1, e_3)\|^2
+\|\mbox{\small$\widetilde{\mu}_7^{t}$}(e_1, e_4)\|^2+
\|\mbox{\small$\widetilde{\mu}_7^{t}$}(e_2, e_3)\|^2+
\|\mbox{\small$\widetilde{\mu}_7^{t}$}(e_2, e_4)\|^2\right)\\
&\hspace{1.1cm}= - 6\left(t+\frac{1}{t}\right), \ \ t\in [-1,0).
\end{align*}

\begin{figure}
\begin{center}
\psset{unit=2cm}
\begin{pspicture}(-2,-1.6)(2,2)
%\grilla
\pscircle(0,0){1}
\psaxes[labels=none,ticks=none]{->}(0,0)(-2,-1.7)(2,2)%
\psline[linewidth=0.5pt](-1,-1)(1,1)
\psarc[linewidth=1.2pt]{[-)}(0,0){1}{45}{90}
\psarc[linewidth=1.2pt]{(-)}(0,0){1}{90}{180}
\psarc[linewidth=1.2pt]{(-]}(0,0){1}{180}{225}
\psdot(0,0)
\rput{-20}(0.45,1.1){$\ngo_4^*$}
\rput{35}(-0.8,0.8){$\ngo_3$}
\rput(-1.1,-0.5){$\ngo_4$}
\rput(0.1,-0.1){$\ngo_1$}
\rput(0.25,1.9){b(t)}
\rput(1.9,-0.2){a(t)}
%\rput(0.3,1.4){$\ngo_2^+$}
\rput(-0.3,1.4){$\ngo_2^+$}
%\psline[linewidth=0.5pt]{->}(0.23,1.33)(0.02,1.1)
\psline[linewidth=0.5pt]{->}(-0.23,1.33)(-0.02,1.1)
%\rput(-1.5,-0.3){$\ngo_5$}
\rput(-1.6,0.4){$\ngo_5,\ngo_2^{-}$}
%\psline[linewidth=0.5pt]{->}(-1.43,-0.23)(-1.1,-0.02)
\psline[linewidth=0.5pt]{->}(-1.3,0.3)(-1.1,0.02)
\end{pspicture}\captionof{figure}{\ Pfaffian forms of $\ngo_1$,\ldots,$\ngo_5$.}\label{grpfaffian}
\end{center}
\end{figure}

\begin{proposition}
\ $\mbox{\small$\widetilde{\mu}_7^{t}$}\notin \RR^{*}\ \U(1)\times\U(1)\times\U(1)\cdot \mu_7^{s}$ \
for all \ $t\in [-1,0), s\in(0,1]$.
\end{proposition}
\begin{proof}
If we suppose that there exists $c\in\RR^{*}$ and $\varphi\in \U(1)\times\U(1)\times\U(1)$ such that
$c\widetilde{\mu}_7^{t} = \varphi\cdot\mu_7^{s}$, then
$\varphi = \left[\begin{smallmatrix}
\varphi_1&&\\&\varphi_2&\\&&\varphi_3
\end{smallmatrix}\right]$
and $c^2\Ric_{\widetilde{\mu}_7^{t}}|_{\zg} = \varphi_3\Ric_{\mu_7^{s}}|_{\zg}\varphi_{3}^{-1}$. Hence
$c^2\left[\begin{smallmatrix}-t&\\&-1/t\end{smallmatrix}\right] =
\left[\begin{smallmatrix}s&\\&1/s\end{smallmatrix}\right]$; taking quotients of their eigenvalues
we deduce that $s^2=t^2$ or $s^2=1/t^2$, which gives \ $t=-s$ if $t\in [-1,0), s\in(0,1]$. \
From this it is enough to prove that for all \ $t\in (0,1], c\in\RR^{*}$,
\begin{align}\label{proof}
\widetilde{\mu}_7^{-t}\notin c\ \U(1)\times\U(1)\times\U(1)\cdot \mu_7^{t}.
\end{align}
Moreover, if $\mbox{\small$\widetilde{\mu}_7^{-t}$}\in c\ \U(1)\times\U(1)\times\U(1)\cdot \mu_7^{t}$,
then $\|\mbox{\small$\widetilde{\mu}_7^{-t}$}\|^2= c^2\|\mu_7^t\|^2$, which yields $c^2=1$, and hence $c=\pm 1$.
Thus it is sufficient to take $c=1$ (if $c=-1$ the equations does not change).

Suppose, contrary to our claim, that $\widetilde{\mu}_7^{-t} = G\cdot \mu_7^{t}$
where
\begin{align*}
\begin{array}{lll}
G=\left[\begin{smallmatrix}
a&-b&&&&\\
b&a&&&&\\
&&c&-d&&\\
&&d&c&&\\
&&&&k&-h\\
&&&&h&k
\end{smallmatrix}\right], &&
G^{-1}=\left[\begin{smallmatrix}
a&b&&&&\\
-b&a&&&&\\
&&c&d&&\\
&&-d&c&&\\
&&&&k&h\\
&&&&-h&k
\end{smallmatrix}\right],
\end{array}
\end{align*}
with \ $a^2+b^2=c^2+d^2=k^2+h^2=1$. We thus get
\begin{align*}
& \bullet \ \widetilde{\mu}_7^{-t}(e_1,e_2)= - \sqrt{t+1/t}e_4 = G\cdot \mu_7^{t}(e_1,e_2) = d\sqrt{t+1/t} e_3 - c\sqrt{t+1/t} e_4.\\
& \bullet \ \widetilde{\mu}_7^{-t}(e_1,e_3) = \sqrt{t}e_5 = G\cdot \mu_7^{t}(e_1,e_3)\\
&\hspace{2.1cm} = \left\{(ac+bd)k\sqrt{t} + (bc-ad)\tfrac{h}{\sqrt{t}}\right\} e_5 + \left\{(ac+bd)h\sqrt{t} + (ad-bc)\tfrac{k}{\sqrt{t}}\right\} e_6.\\
& \bullet \ \widetilde{\mu}_7^{-t}(e_1,e_4) = \frac{1}{\sqrt{t}}e_6 = G\cdot \mu_7^{t}(e_1,e_4)\\
&\hspace{2.1cm} = \left\{(ad-bc)k\sqrt{t} + (ac+bd)\tfrac{h}{\sqrt{t}}\right\} e_5 + \left\{(ad-bc)h\sqrt{t} - (ac+bd)\tfrac{k}{\sqrt{t}}\right\} e_6.
\end{align*}
This is equivalent at next system (the other tree brackets produce the same equations):
\begin{equation*}
\left\{
\begin{aligned}
&c=1, \ d=0,\\
&a=k,\\
&b-ht=0,\\
&a=-k,\\
&h+bt=0,
\end{aligned}
\right.
\end{equation*}
It follows easily that $a=b=0$, contrary to $a^2+b^2=1$. Since $G$ was arbitrary, (\ref{proof})
is proved.
\end{proof}

%================================================================================================

\section{Results obtained in dimension eight}\label{result8}

In this section, our aim is to exhibit many families depending on one (see Example \ref{ejm1par} and
Example \ref{ejm1par2}), two (see Example \ref{ejm2par} and Example \ref{ejm2par2}) and three
(see Example \ref{ejm3par}) parameters of abelian complex structures on $8$-dimensional $2$-step
nilpotent Lie algebras, by using that they all admit minimal metrics for the types $(1<2;4,4)$ and $(1<2;6,2)$.

Following the idea developed in dimension six, we will determine the quotients \linebreak
$P_{4,2}(\RR)/\RR_{>0}\U(2)$ and $P_{2,3}(\RR) /\RR_{>0}\U(1)$ in the cases $(4,4)$ and $(6,2)$
respectively. This may be viewed as a first step towards the classification of abelian complex structures on
$8$-dimensional nilmanifolds. From now on, we keep the notation used in \cite{praga}.

\subsection{Type (4,4)}\label{tip44}
In this case $\vg_1=\RR^4$ and $\vg_2=\RR^4$, \ and we consider the vector space \linebreak
$W :=\Lambda^2\vg_1^*\otimes\vg_2$. If $\{ X_1,\ldots,X_4, Z_1,\ldots,Z_4\}$ is a basis of
$\ngo$ such that $\vg_1=\la X_1,...,X_4\ra_{\RR}$ and $\vg_2=\la Z_1,\ldots,Z_4\ra_{\RR}$,
then each element in $W$ will be described as
\begin{align*}
& \mu(X_1,X_2)= a_1 Z_1 + a_2 Z_2 + a_3 Z_3 + a_4 Z_4, & &\mu(X_1,X_3)= b_1 Z_1 + b_2 Z_2 + b_3 Z_3 + b_4 Z_4,\\
& \mu(X_1,X_4)= c_1 Z_1 + c_2 Z_2 + c_3 Z_3 + c_4 Z_4, & &\mu(X_2,X_3)= d_1 Z_1 + d_2 Z_2 + d_3 Z_3 + d_4 Z_4,\\
& \mu(X_2,X_4)= e_1 Z_1 + e_2 Z_2 + e_3 Z_3 + e_4 Z_4, & &\mu(X_3,X_4)= f_1 Z_1 + f_2 Z_2 + f_3 Z_3 + f_4 Z_4.
\end{align*}
The complex structure and the compatible metric will be always defined by
$$
\begin{array}{lll}
J=\left[\begin{smallmatrix}
0&-1&&&&&&\\
1&0&&&&&&\\
&&0&-1&&&&\\
&&1&0&&&&\\
&&&&0&-1&&\\
&&&&1&0&&\\
&&&&&&0&-1\\
&&&&&&1&0
\end{smallmatrix}\right], && \la X_i,X_j\ra= \la Z_i,Z_j\ra=\delta_{ij}.
\end{array}
$$
If $A=(a_1,\ldots,a_4), \ \ldots \ , F=(f_1,\ldots,f_4)$,
then $J$ is integrable on $N_\mu$ (i.e. $J$ satisfies (\ref{integral})), $\mu\in W$, if and only if
\begin{align}\label{muint}
E=B+JC+JD,
\end{align}
and $J$ is abelian if and only if
\begin{align}\label{muabel}
& E=B, \qquad D=-C.
\end{align}

Define $v_i=(a_i,b_i,c_i,d_i,e_i,f_i)$, \ $i=1,2,3,4$. It is easy to check that for any $\mu\in W$,\linebreak
$\Ric_{\mu}\mid_{\vg_2}=\frac{1}{2}[\la v_i,v_j\ra]$, $1\leq i,j \leq 4$, and \vspace{0.3cm}
\begin{align*}
\Ric_{\mu}\mid_{\vg_1}=-\frac{1}{2}
\left[\begin{smallmatrix}
\|A\|^2+\|B\|^2+\|C\|^2&\la B,D\ra+\la C,E\ra&
-\la A,D\ra+\la C,F\ra&-\la A,E\ra-\la B,F\ra\\
\la B,D\ra+\la C,E\ra&\|A\|^2+\|D\|^2+\|E\|^2&
\la A,B\ra+\la E,F\ra&\la A,C\ra-\la D,F\ra\\
-\la A,D\ra+\la C,F\ra&\la A,B\ra+\la E,F\ra&
\|B\|^2+\|D\|^2+\|F\|^2&\la B,C\ra+\la D,E\ra\\
-\la A,E\ra-\la B,F\ra&\la A,C\ra-\la D,F\ra&\la B,C\ra+\la D,E\ra&
\|C\|^2+\|E\|^2+\|F\|^2
\end{smallmatrix}\right].
\end{align*}
Therefore
{\small
\begin{align*}
\Ricc_{\mu}\mid_{\vg_1}=\frac{1}{4}
\begin{bmatrix}
-\alpha&0&\la A+F,D-C\ra&\la A+F,B+E\ra\\
0&-\alpha&-\la A+F,B+E\ra&\la A+F,D-C\ra\\
\la A+F,D-C\ra&-\la A+F,B+E\ra&-\beta&0\\
-\la A+F,B+E\ra&\la A+F,D-C\ra&0&-\beta
\end{bmatrix},
\end{align*}}

{\small
\begin{align*}
\Ricc_{\mu}\mid_{\vg_2}=\frac{1}{4}
\begin{bmatrix}
\|v_1\|^2+\|v_2\|^2&0&\la v_1,v_3\ra+\la v_2,v_4\ra&
\la v_1,v_4\ra-\la v_2,v_3\ra\\
0&\|v_1\|^2+\|v_2\|^2&\la v_2,v_3\ra-\la v_2,v_4\ra&
\la v_2,v_4\ra+\la v_1,v_3\ra\\
\la v_1,v_3\ra+\la v_2,v_4\ra&\la v_2,v_3\ra-\la v_2,v_4\ra&\|v_3\|^2+\|v_4\|^2&0\\
\la v_1,v_4\ra-\la v_2,v_3\ra&\la v_2,v_4\ra+\la v_1,v_3\ra&0&
\|v_3\|^2+\|v_4\|^2
\end{bmatrix},
\end{align*}}
where \ {\small$\alpha:=2\|A\|^2+\|B\|^2+\|C\|^2+\|D\|^2+\|E\|^2$} and
{\small$\beta:=\|B\|^2+\|C\|^2+\|D\|^2+\|E\|^2+2\|F\|^2$}.

One type of minimality which is easy to characterize is $(1<2;4,4)$. Indeed,
if for any $\mu\in W $ we have that $\Ricc_{\mu}\mid_{\vg_1}=pI_{\scriptscriptstyle 4}$ and
$\Ricc_{\mu}\mid_{\vg_2}=qI_{\scriptscriptstyle 4}$, then
\begin{align*}
\Ricc_{\mu} =
\left[\begin{smallmatrix}
pI_{\scriptscriptstyle 4}&\\
&qI_{\scriptscriptstyle 4}
\end{smallmatrix}\right] =
(2p-q)I_{\scriptscriptstyle 8} + (q-p)
\left[\begin{smallmatrix}
I_{\scriptscriptstyle 4}&\\
&2I_{\scriptscriptstyle 4}
\end{smallmatrix}\right] \in \RR I + \Der(\mu).
\end{align*}
The following are sufficient conditions for any $\mu\in W$ is minimal of type
$(1<2;4,4)$.
\begin{enumerate}
\item[(i)] Conditions for $\Ricc_{\mu}\mid_{\ngo_1}\in \RR I$:
\begin{enumerate}[$\bullet$]
\item $\la A+F,D-C\ra =0$.
\item $\la A+F,B+E\ra = 0$.
\item $\|A\|^2 = \|F\|^2$.
\end{enumerate}
\item[(ii)] Conditions for $\Ricc_{\mu}\mid_{\ngo_2}\in \RR I$:
\begin{enumerate}[$\bullet$]
\item $\|v_1\|^2 + \|v_2\|^2 = \|v_3\|^2 + \|v_4\|^2$.
\item $\la v_1,v_3\ra = -\la v_2,v_4\ra$.
\item $\la v_1,v_4\ra = \la v_2,v_3\ra$.
\end{enumerate}
\end{enumerate}
Moreover, if $\mu$ satisfies the conditions given in $(\mathrm{i})$ and $(\mathrm{ii})$,
we obtain \ $p=-\frac{1}{4}\alpha$ \ and \linebreak $q=\frac{1}{4}\left(\|v_1\|^2 +\|v_2\|^2\right)$.

In the rest of this section we will study the Pfaffian forms of $\mu\in W$.
Since $\dim{\vg_1} = \dim{\vg_2} = 4$, it follows that the Pfaffian form of any
$\mu\in W$ belongs to the set $P_{4,2}(\RR)$; so the goal is to determine the quotient
$P_{4,2}(\RR)/\RR_{>0}\U(2)$. As in the case $(4, 2)$ there is the identification
$f(\mu)\in P_{4,2}(\RR) \leftrightarrow A_{f}$, where
$A_{f}$ is a symmetric matrix, and, in consequence,
\begin{align}\label{cocient1}
P_{4,2}(\RR) / \pm \Gl_4(\RR) = \left\{
\left[\begin{smallmatrix}
1&&&\\
&1&&\\
&&-1&\\
&&&0\\
\end{smallmatrix}\right],\ldots\right\}.
\end{align}
Based on the classification of complex metabelian (two-step nilpotent) Lie algebras in \linebreak
dimension up to 9 given by L. Yu. Galitski and D. A. Timashev in \cite{metabel}, and by using the
identifications of the real forms of Lie algebra on $\CC$, we have
\begin{equation}\label{cocient}
\begin{aligned}
P_{4,2}(\CC) /\Gl_2(\CC) \ = \
\end{aligned}
\left\{
\begin{aligned}
&x^2-y^2-z^2+w^2\\
&x^2-y^2-z^2\\
&x^2-y^2\\
&x^2\\
&0
\end{aligned}
\right. \ \simeq \
\left\{
\begin{aligned}
&+ \ + \ + \ +\\
&+ \ + \ + \ 0\\
&+ \ + \ \ 0 \ \ 0\\
&+ \ \ 0 \ \ 0 \ \ 0\\
& \ 0 \ \ 0 \ \ 0 \ \ 0
\end{aligned}
\right.
\end{equation}

\begin{remark}
The polynomial $f = x^2+y^2+z^2+w^2$ is not the Pfaffian form of any $\mu\in W$.
In general, $f>0$ ($\Leftrightarrow$ $J_Z$ are invertible $\forall Z$)
is not the Pfaffian form of any $\mu\in W$. The dimensions allowed for this are: $(2k,1)$,
$(4k,2)$, $(4k,3)$, $(8k,4)$,\ \ldots \ , $(8k,7)$, $(16k,8)$, $(32k,9)$.
\end{remark}

The following expression was obtained by direct calculation rather than the equations
(\ref{cocient1}) and (\ref{cocient}).

{\small\begin{equation*}
\begin{aligned}
P_{4,2}(\RR) / \U(2) \ \simeq \ \sym(4) / \U(2) \ &= \
\left\{
\begin{aligned}
&\left(a I,
\left[\begin{smallmatrix}
b&&&\\
&-b&&\\
&&c&\\
&&&-c\\
\end{smallmatrix}\right]
\right); & & a, b, c\in \RR.\\
&\\
&\left(\left[\begin{smallmatrix}
a&&&\\
&a&&\\
&&b&\\
&&&b\\
\end{smallmatrix}\right],
\left[\begin{smallmatrix}
c&h&&\\
h&-c&&\\
&&d&l\\
&&l&-d\\
\end{smallmatrix}\right]
\right); & & a, b, c, d, h, l \in \RR \ (a<b).
\end{aligned}
\right. \\
&\\
&= \
{\small
\left\{
\begin{aligned}
&ax^2+by^2+cz^2+dw^2; & & a+b=c+d,\\
& & & a, b, c, d\in\RR.\\
&&&\\
&ax^2+by^2+cz^2+dw^2+hxy+lzw; \ \ & & a+b < c+d,\\
& & & a, b, c, d, h, l\in \RR.
\end{aligned}
\right.}
\end{aligned}
\end{equation*}}

In what follows, we given some curves and families of minimal metrics of type \linebreak
$(1<2;4,4)$, which Pfaffian forms appear in the above quotient.

\begin{example}\label{ejm3par}
Let $\mu_{\scriptscriptstyle krst}\in W$ be given by
\begin{align*}
& A=(s,t,0,0),\ \qquad B=(0,0,r,0),\\
& C=(0,0,0,k), \qquad D=(0,0,0,-k),\\
& E=(0,0,r,0), \qquad F=(s,-t,0,0),
\end{align*}
with $k, r, s, t \in \RR$. \ It is clear that $\mu_{\scriptscriptstyle krst}$ satisfies (\ref{muint})
and (\ref{muabel}), and hence $(N_{\mu_{\scriptscriptstyle krst}}, J)$ is an abelian complex
nilmanifold for all $k, r, s, t \in \RR$. Furthermore, if $k^2+r^2=s^2+t^2$ then the family
$\{(N_{\mu_{\scriptscriptstyle krst}}, J, \ip) : k^2+r^2=s^2+t^2\}$ of minimal
(conditions (i) and (ii)) metrics is pairwise non-isometric, up to scaling. This gives rise
then a 3-parameter family of pairwise non-isomorphic abelian complex nilpotent Lie groups
(see Theorem \ref{RS}). On the other hand, the Pfaffian form of
$\mu_{\scriptscriptstyle krst}$ is
\begin{align*}
f(\mu_{\scriptscriptstyle krst}) = s^2 x^2 - t^2 y^2 - r^2 z^2 - k^2 w^2.
\end{align*}
\end{example}
%En este caso, el unico $\Sl_4(\RR)$-invariante polinomial es $\det(A_{f}) = -(krst)^2$, donde $A_{f}$
%es la matriz asociada a $f(\mu_{\scriptscriptstyle krst})$. Notemos que si alguno de los
%parametros es cero entonces $\det(A_{f})=0$, lo cual implica la importancia del $\U(2)$-invariante.

\begin{example}\label{ejm2par}
Let $\lambda_{\scriptscriptstyle rst}$ be defined by:
\begin{align*}
& A=(0,r,0,0),\ \qquad B=(0,0,s,0),\\
& C=(0,0,0,t), \qquad D=(0,0,0,-t),\\
& E=(0,0,s,0), \qquad F=(0,-r,0,0),
\end{align*}
where $r, s, t \in \RR$. \ We have $(N_{\lambda_{\scriptscriptstyle rst}}, J)$ is an abelian complex
nilmanifold for all $r, s, t$ in $\RR$. \ If $r^2=s^2+t^2$ then the family
$\{(N_{\lambda_{\scriptscriptstyle rst}}, J, \ip) : r^2=s^2+t^2\}$ of minimal compatible
metrics is pairwise non-isometric, unless scalar multiples. This gives rise then a 2-parameter
family of pairwise non-isomorphic abelian complex nilpotent Lie groups. Note that the Pfaffian
form of $\lambda_{\scriptscriptstyle rst}$ is
$$f(\lambda_{\scriptscriptstyle rst}) = - r^2 y^2 - s^2 z^2 - t^2 w^2.$$
\end{example}

\begin{example}\label{ejm1par}
Let $\nu_{\scriptscriptstyle st}$ be given by $A=(s,0,0,0)= -F$, \ $B = E = 0$,
$C=(0,0,t,0) = -D$, with $s, t \in \RR$. \ Therefore, $(N_{\nu_{\scriptscriptstyle st}}, J)$
is an abelian complex nilmanifold for all $s, t \in \RR$. \ Furthermore, if $s^2=t^2$ then
the curve $\{\nu_{\scriptscriptstyle st} : s^2=t^2\}$ of minimal compatible
metrics is pairwise non-isometric, unless scalar multiples. This gives a curve of pairwise
non-isomorphic abelian complex nilpotent Lie groups. \ Finally, the Pfaffian form of
$\nu_{\scriptscriptstyle st}$ is
$$f(\nu_{\scriptscriptstyle st}) = - s^2 x^2 - t^2 z^2.$$
\end{example}

\begin{example}\label{ejm2par2}
Let $\mu_{\scriptscriptstyle rst}\in W$ be defined by:
\begin{align*}
& A=(r,0,0,0),\ \qquad B=(0,0,s,0),\\
& C=(0,0,0,t), \qquad D=(0,0,0,-t),\\
& E=(0,0,s,0), \qquad F=(0,r,0,0),
\end{align*}
where $r, s, t \in \RR$. Hence $(N_{\mu_{\scriptscriptstyle rst}}, J)$ is an abelian complex
nilmanifold for all $r, s, t \in \RR$. If $r^2=s^2+t^2$ then
$\{\mu_{\scriptscriptstyle rst} : r^2=s^2+t^2\}$ of minimal compatible
metrics is pairwise non-isometric, up to scaling. This gives rise then a 2-parameter
family of pairwise non-isomorphic abelian complex nilpotent Lie groups. Note that
the Pfaffian form of $\mu_{\scriptscriptstyle rst}$ is given by
$$f(\mu_{\scriptscriptstyle rst}) = r^2 xy - s^2 z^2 - t^2 w^2.$$
\end{example}

%------------------------------------------------------------------------------------------------

\subsection{Type (6,2)}\label{tip62}
For $\vg_1=\RR^6$ and $\vg_2=\RR^2$, consider $\widetilde{W} :=\Lambda^2\vg_1^*\otimes\vg_2$.
Fix basis  $\{ X_1,\ldots,X_6\}$ and $\{Z_1,Z_2\}$ of $\vg_1$ and $\vg_2$, respectively. Each
element $\mu\in \widetilde{W}$ will be described as
\begin{align*}
& \mu(X_1,X_2)= a_1 Z_1 + a_2 Z_2, & & \mu(X_1,X_3)= b_1 Z_1 + b_2 Z_2, & & \mu(X_1,X_4)= c_1 Z_1 + c_2 Z_2,\\
& \mu(X_1,X_5)= d_1 Z_1 + d_2 Z_2, & & \mu(X_1,X_6)= e_1 Z_1 + e_2 Z_2, & & \mu(X_2,X_3)= f_1 Z_1 + f_2 Z_2,\\
& \mu(X_2,X_4)= g_1 Z_1 + g_2 Z_2, & & \mu(X_2,X_5)= h_1 Z_1 + h_2 Z_2, & & \mu(X_2,X_6)= i_1 Z_1 + i_2 Z_2,\\
& \mu(X_3,X_4)= k_1 Z_1 + k_2 Z_2, & & \mu(X_3,X_5)= l_1 Z_1 + l_2 Z_2, & & \mu(X_3,X_6)= m_1 Z_1 + m_2 Z_2,\\
& \mu(X_4,X_5)= n_1 Z_1 + n_2 Z_2, & & \mu(X_4,X_6)= p_1 Z_1 + p_2 Z_2, & & \mu(X_5,X_6)= q_1 Z_1 + q_2 Z_2.\\
\end{align*}
The complex structure and the compatible metric will be always defined by
$$
\begin{array}{lll}
\begin{matrix}
JX_1=X_2,&JX_3=X_4,\\
JX_5=X_6,&JZ_1=Z_2.
\end{matrix} && \la X_i,X_j\ra=\delta_{ij}, \ \la Z_k,Z_l\ra=\delta_{kl}.
\end{array}
$$
If $A=(a_1,a_2), \ \ldots \ , Q=(q_1,q_2)$, then  $J$ satisfies (\ref{integral}) if and only if
\begin{align}\label{muint2}
G=B+JC+JF, & & I=D+JE+JH, & & P=L+JM+JN,
\end{align}
and $J$ is abelian if and only if
\begin{align}\label{muabel2}
& B=G, \quad C=-F, \quad D=I, \quad E=-H, \quad L=P, \quad M=N.
\end{align}
Let $v_i=(a_i,b_i,c_i,d_i,e_i,f_i)$, \ $i=1,2$. \ It follows that
$\Ric_{\mu}\mid_{\vg_2}=\frac{1}{2}[\la v_i,v_j\ra]$, $1\leq i,j \leq 2$, and \vspace{0.3cm}
\begin{align*}
\Ricc_{\mu}\mid_{\vg_2}={\small \frac{1}{4}
\begin{bmatrix}
\|v_1\|^2+\|v_2\|^2 & 0\\
0&\|v_1\|^2+\|v_2\|^2
\end{bmatrix}} \ \in \ \RR I.
\end{align*}
For the complicated expressions, we only give sufficient conditions for any $\mu\in \widetilde{W}$ is minimal
of type $(1<2;6,2)$ when $J$ is abelian.
\begin{enumerate}\label{condRI}
\item [($\dag$)] Conditions for $\Ricc_{\mu}\mid_{\ngo_1}\in \RR I$:
{\small\begin{enumerate}
\item $\|A\|^2 = \|K\|^2 = \|Q\|^2$, \quad
$\|B\|^2+\|C\|^2=\|D\|^2+\|E\|^2=
\|L\|^2+\|M\|^2$.
\item $\la A+K, B \ra = \la A+K, C \ra = \la A+Q, D \ra = \la A+Q, E \ra = \la K+Q, L \ra =
\la K+Q, M \ra = 0$.
\item $\la B, L \ra = - \la C, N \ra$, \ $\la B, M \ra = - \la C, P \ra$,
\ $\la B, D \ra = - \la C, E \ra$.
\item $\la C, D \ra = - \la G, H \ra$, \ $\la D, L \ra = - \la E, M \ra$,
\ $\la H, L \ra = - \la I, M \ra$.
\end{enumerate}}
\end{enumerate}
If $\mu$ satisfies the conditions given in $(\dag)$ we thus get
$q:=\frac{1}{4}\left(\|v_1\|^2 +\|v_2\|^2\right)$ and \linebreak
{\small$p:=-\frac{1}{2}\left({\small \|A\|^2+\|B\|^2+\|C\|^2+\|D\|^2+\|E\|^2}\right)$}.

Since $\dim{\vg_1} = 6$ and $\dim{\vg_2} = 2$, the Pfaffian form of any $\mu\in \widetilde{W}$
belongs to the set $P_{2,3}(\RR)$. Unlike the previous two cases, there is no identification of
$f(\mu)\in P_{2,3}(\RR)$ with a matrix, but it is known that every polynomial in $P_{2,3}(\RR)$
is the Pfaffian form of some $\mu\in \widetilde{W}$ (see \cite{pfaff}). Again, of \cite{metabel},
we obtain
\begin{equation}\label{cocient2}
\begin{aligned}
P_{2,3}(\CC) /\Gl_2(\CC) \ = \
\end{aligned}
\left\{
\begin{aligned}
&x^3\\
&x^2y+xy^2 = xy(x+y)\\
&x^3+x^2y = x^2(x+y) \simeq x^2y
\end{aligned}
\right.
\end{equation}
But it is easy to see that
\begin{align*}
&P_{2,3}(\RR) \cap \Gl_2(\CC)\cdot x^3 = \Gl_2(\RR)\cdot x^3,\\
&P_{2,3}(\RR) \cap \Gl_2(\CC)\cdot (x^2y+xy^2) = \Gl_2(\RR)\cdot (x^2y+xy^2),\\
&P_{2,3}(\RR) \cap \Gl_2(\CC)\cdot x^2y = \Gl_2(\RR)\cdot x^2y,
\end{align*}
and therefore
\begin{equation*}
\begin{aligned}
P_{2,3}(\RR) / \Gl_2(\RR) \ = \
\end{aligned}
\left\{
\begin{aligned}
&x^3\\
&x^2y+xy^2 = xy(x+y)\\
&x^3+x^2y = x^2(x+y) \simeq x^2y
\end{aligned}
\right.
\end{equation*}
\begin{example}
Let $\mu_{\scriptscriptstyle st}^1, \mu_{\scriptscriptstyle st}^2,
\mu_{\scriptscriptstyle st}^3\in \widetilde{W}$ be defined by: \ for all $s, t \in \RR$,
\begin{align*}
& A^1=(0, s), &&  A^2=(t, 0), && A^3=(0, s),\\
& E^1=(t, 0), &&  K^2=(0, s), && E^3=(t, 0),\\
& H^1=(-t, 0), && Q^2=(s, t). && H^3=(-t, 0),\\
& K^1=(s, 0). &&   && K^3=(s, 0),\\
&   &&   && Q^3=(-s, 0).
\end{align*}
It follows immediately that $(N_{\mu_{\scriptscriptstyle st}^1}, J)$,
$(N_{\mu_{\scriptscriptstyle st}^2}, J)$ and $(N_{\mu_{\scriptscriptstyle st}^3}, J)$
are abelian complex nil\-ma\-ni\-folds for all $s, t \in \RR$, as they satisfy (\ref{muint2}) and
(\ref{muabel2}). Furthermore, they are not minimal of type $(1<2; 6,2)$ and its
Pfaffian forms are given by
\begin{align*}
& f(\mu_{\scriptscriptstyle st}^1) = st^2 x^3, \quad
f(\mu_{\scriptscriptstyle st}^2) = s^2t x^2y + st^2 xy^2, \quad
f(\mu_{\scriptscriptstyle st}^3) = st^2 x^3 + s^3 x^2y.
\end{align*}
Hence $\{(N_{\mu_{\scriptscriptstyle st}^2}, J) : s,t\in\RR\smallsetminus\{0, \pm 1\} \}$ and
$\{(N_{\mu_{\scriptscriptstyle st}^3}, J) : s,t\in\RR\smallsetminus\{0\}\}$ are curves of
abelian complex nilmanifolds, which is due to the fact that
\begin{align*}
& \forall a, b \in \RR\smallsetminus\{0, \pm 1\}, a\neq b: \
x^2y+axy^2 \ \notin \ \U(1)\cdot(x^2y+bxy^2).\\
& \forall a, b \in \RR\smallsetminus\{0\}, a\neq b: \
x^3+ax^2y \ \notin \ \U(1)\cdot(x^3+bx^2y).
\end{align*}
\end{example}
\begin{remark} Let $p(x,y) = \sum_{i=0}^{3}a_{i}x^{3-i}y^{i} \in P_{2,3}(\RR)$. Define
\begin{align*}
& \triangle(p):= (3a_0 + a_2)^2 + (a_1 + 3a_3)^2.\\
& \|p\|^2:= 6a_0^2 + 2a_1^2 + 2a_2^2 + 6a_3^2.\\
& D(p):= 18a_0 a_1 a_2 a_3 + a_1^2 a_2^2 - 4a_0 a_2^3 - 4a_1^3 a_3 - 27a_0^2 a_3^2.
\end{align*}
We have $\triangle$ is $\SO(2)$-invariant, $\|\cdot\|^2$ is $\Or(2)$-invariant and
$D$ is $\Sl_2(\RR)$-invariant.

Note that using quotients of the above invariants we can also obtain that
$\{(N_{\mu_{\scriptscriptstyle st}^2}, J) : s,t\in\RR\smallsetminus\{0, \pm 1\} \}$ and
$\{(N_{\mu_{\scriptscriptstyle st}^3}, J) : s,t\in\RR\smallsetminus\{0\}\}$ are curves of
abelian complex nilmanifolds.
\end{remark}

\begin{example}\label{ejm1par2}
Let $\lambda_{\scriptscriptstyle st}\in \widetilde{W}$ be given by $A=(t,s)$, $K=(-s,t)$ and
$Q=(s,t)$, with $s, t \in \RR$. \ We obtain $(N_{\lambda_{\scriptscriptstyle st}}, J)$ is an
abelian complex nilmanifold for all $s, t \in \RR$. Furthermore,
$\lambda_{\scriptscriptstyle st}$ is minimal of type $(1<2; 6,2)$. \
On the other hand, the Pfaffian form of $\lambda_{\scriptscriptstyle st}$ is
\begin{align*}
f(\lambda_{\scriptscriptstyle st}) = s^2t x^3 + s^3 x^2y - t^3 xy^2 - st^2 y^3.
\end{align*}
Define $a:=s^2$, $b:=t^2$, and consider
$$h(a,b):= \frac{D(f(\lambda_{\scriptscriptstyle st}))}{(\triangle(f(\lambda_{\scriptscriptstyle st})))^2} = \frac{4ab(a^2-b^2)^2}{(a+b)^6}$$
If $a+b=1$ then $h(a)= 4a(1-a)(2a-1)^2$ is an injective function for all $a\geq 1$. Hence $\{\lambda_{\scriptscriptstyle st} : s^2 + t^2 = 1, s\geq 1\}$ is a curve of pairwise
non-isometric metrics.
\end{example}

%=================================================================================================

\end{document}